\pdfoutput=1
\RequirePackage{ifpdf}
\ifpdf 
\documentclass[pdftex]{sigma}
\else
\documentclass{sigma}
\fi

\numberwithin{equation}{section}

\newtheorem{Theorem}{Theorem}[section]
\newtheorem{Corollary}[Theorem]{Corollary}
\newtheorem{Lemma}[Theorem]{Lemma}
\newtheorem{Proposition}[Theorem]{Proposition}
 { \theoremstyle{definition}
\newtheorem{Remark}[Theorem]{Remark}
\newtheorem*{Notation}{Notation}
}

\newcommand{\model}{\automorphism({TM},{\mathfrak g})}
\newcommand{\varcheck}[1]{{#1^\vee}}
\DeclareMathOperator{\pseud}{pseud}
\DeclareMathOperator{\rank}{rank}
\DeclareMathOperator{\adjoint}{ad}
\DeclareMathOperator{\Adjoint}{Ad}
\DeclareMathOperator{\kernel}{ker}
\DeclareMathOperator{\automorphism}{Aut}
\DeclareMathOperator{\cokernel}{coker}
\DeclareMathOperator{\curvature}{curv}
\newcommand\monomorphism{\hookrightarrow}
\newcommand\identity{\mathrm{id}}
\newcommand{\modulo}{~\mathrm{mod}~}

\newenvironment{aside}{\begin{quote}\em }{\end{quote}}

\begin{document}


\newcommand{\arXivNumber}{1605.04365}

\renewcommand{\PaperNumber}{114}

\FirstPageHeading

\ShortArticleName{Cartan Connections on Lie Groupoids and their Integrability}

\ArticleName{Cartan Connections on Lie Groupoids\\ and their Integrability}

\Author{Anthony D.~BLAOM}

\AuthorNameForHeading{A.D.~Blaom}

\Address{10 Huruhi Road, Waiheke Island, New Zealand}
\Email{\href{mailto:anthony.blaom@gmail.com}{anthony.blaom@gmail.com}}

\ArticleDates{Received May 19, 2016, in f\/inal form December 02, 2016; Published online December 07, 2016}

\Abstract{A multiplicatively closed, horizontal $n$-plane f\/ield $D$ on a Lie groupoid~$G$ over~$M$ generalizes to intransitive geometry the classical notion of a Cartan connection. The inf\/initesimalization of the connection~$D$ is a Cartan connection~$\nabla $ on the Lie algebroid of~$G$, a notion already studied elsewhere by the author. It is shown that $\nabla $ may be regarded as inf\/initesimal parallel translation in the groupoid $G$ along $D$. From this follows a proof that $D$ def\/ines a pseudoaction generating a~pseudogroup of transformations on~$M$ precisely when the curvature of~$\nabla $ vanishes. A byproduct of this analysis is a detailed description of multiplication in the groupoid~$J^1 G$ of one-jets of bisections of~$G$.}


\Classification{53C05; 58H05; 53C07}

\section{Introduction}\label{intro}
This article describes a generalization of classical Cartan connections using the language of Lie groupoids, with which the reader is assumed to have some familiarity. We recommend the introduction given in~\cite{CannasdaSilva_Weinstein_99}. For more detail, see~\cite{Crainic_Fernandes_11,Dufour_Nguyen_05,Mackenzie_05}.

The present article provides some detail missing from our work published on related matters~\cite{Blaom_06,Blaom_12,Blaom_13,Blaom_16} and is more technical than those works. Applications of the Lie groupoid approach to Cartan connections will be given elsewhere.

\subsection{Cartan connections on Lie groupoids}\label{onepoint}
Let $G$ be a Lie groupoid over a smooth connected manifold $M$ and $J^1 G$ the corresponding Lie groupoid of one-jets of local bisections of $G$. Then a right-inverse $S \colon G \rightarrow J^1 G$ for the canonical projection $J^1 G \rightarrow G$ determines a $n$-plane f\/ield $D \subset TG$ on~$G$, where $n=\dim M$. If $S \colon G \rightarrow J^1 G$ is additionally a morphism of Lie groupoids, we call $D$ (or $S$) a {\it Cartan connection} on~$G$~\cite{Blaom_12,Blaom_16}. In the general terminology of~\cite{del_Hoyo_Fernandes_16}, $S$ is a unital f\/lat cleavage for the Lie groupoid f\/ibration $J^1 G \rightarrow G$. Or, if we view the tangent bundle $TG$ as a Lie groupoid over $TM$, then $D \subset TG$ should be multiplicatively closed in the sense of~\cite{Crainic_etal_15, Salazar_13}.

Cartan connections on Lie groupoids are the global analogues of the Cartan connections on Lie algebroids f\/irst introduced in~\cite{Blaom_06}~-- here called {\it infinitesimal Cartan connections}~-- and studied further in~\cite{Blaom_12,Blaom_13}. Inf\/initesimal Cartan connections also appear in a new approach to Cartan geometries introduced in the recent monograph \cite{Crampin_Saunders_16} (where they are called Blaom connections).

Cartan connections in various incarnations have been central to the study of f\/inite-type geometric structures, and to the geometric approach to studying overdetermined systems of partial dif\/ferential equations. The f\/irst example can be found in \`Elie Cartan's seminal paper~\cite{Cartan_10}. In the present article a {\it classical Cartan connection} is a certain one-form on a principal $H$-bundle~$P$ taking values in a vector space $V$ with~$\dim V= \dim P$, as f\/irst formulated by Ehresmann~\cite{Ehresmann_51}. See Section~\ref{classical} for the def\/inition and~\cite{Sharpe_97} for a nice introduction. Our f\/irst result, proven in Sections~\ref{ugt} and~\ref{recover}, is a correspondence between classical Cartan connections and Cartan connections on {\em transitive} Lie groupoids:
\begin{Theorem}\label{theorem1.1}
Every classical Cartan connection $\omega$ on a principal $H$-bundle $P$ determines a canonical Cartan connection $S^\omega$ on the corresponding gauge groupoid $G_P=(P \times P)/H$. Conversely, every Cartan connection on a transitive Lie groupoid $G$ coincides with $S^\omega$, for some classical Cartan connection $\omega $, after the identification $G \cong G_P$ determined by a choice of source-fibre $P \subset G$.
\end{Theorem}

It should be noted here that our notion of a classical Cartan connection is, by default, stripped of any `model data' because such data plays no role at all in the above correspondence\footnote{Model data {\em is} part of Ehresmann's original formulation.}.

It was shown already in \cite{Blaom_06} that every classical Cartan connection on a principal bundle $P$ determines an inf\/initesimal Cartan connection on the associated Atiyah Lie algebroid $TP/H$.

\subsection{Examples}\label{reg}
\looseness=-1 To illustrate the utility of Cartan connections in a simple example, suppose $M$ is a Riemannian manifold, and let $M \times M$ be the pair groupoid. Regard elements of $J^1 (M \times M)$ as one-jets of local transformations of $M$ and let $G \subset J^1 (M \times M)$ be the subgroupoid of all one-jets $J^1 _m \phi $ for which the tangent map $T_m\phi \colon T_{m}M \rightarrow T_{\phi(m)}M$ is an isometry. Then it can be shown (see the appendix) that $G$ carries a canonical Cartan connection $D$ whose signif\/icance is this: The f\/irst-order extension $b=J^1 \phi $ of every local isometry $\phi$ of $M$ is a~local bisection of $G$ integrating~$D$ (i.e., $D$ is tangent to $b$, when $b$ is viewed as a submanifold of $G$). And, conversely, {\em every $n$-dimensional integral manifold of $D$ is locally a local bisection of $G$ arising as the extension of a~local isometry~$\phi $}. In other words, the local isometries are in one-to-one correspondence with the suf\/f\/iciently localised $n$-dimensional integral manifolds of~$D$. In particular, obstructions to the involutivity (integrability) of $D$ amount to obstructions to the existence of local isometries of~$M$.

The main result of the present article, formulated in Section \ref{themain} below, concerns the inf\/ini\-te\-si\-mal (Lie algebroid) version of obstructions to the involutivity of a Cartan connection.

The prototype of an involutive Cartan connection $D$ is the one associated with the action of a Lie group $G_0$ on $M$. In this case one takes $G$ to be the action groupoid $G_0 \times M$ (transitive only if the action is transitive) and def\/ines a Cartan connection as follows: $S(g,m)$ is the one-jet at $m$ of the constant bisection $m' \mapsto (g,m')$.

More examples of Cartan connections on Lie groupoids will be given elsewhere.

\subsection{Inf\/initesimal Cartan connections}\label{hsd}
Inf\/initesimalizing a Cartan connection, viewed as a Lie groupoid morphism $S \colon G \rightarrow J^1 G$, we obtain a splitting $s \colon {\mathfrak g} \rightarrow J^1 {\mathfrak g} $ of the exact sequence
\begin{gather*}
 0\longrightarrow T^*\!M\otimes{\mathfrak g}\longrightarrow J^1{\mathfrak g}\longrightarrow{\mathfrak g}\longrightarrow 0.
\end{gather*}
Here $J^1 {\mathfrak g} $ denotes the Lie algebroid of one-jets of sections of the Lie algebroid ${\mathfrak g} $ of $G$. (A natural
identif\/ication of the Lie algebroid of $J^1 G$ with $J^1 {\mathfrak g} $ is recalled in Section~\ref{repsagain} below.) In the category of vector bundles, splittings $s$ of the above sequence are in one-to-one correspondence with Koszul connections~$\nabla$ on~${\mathfrak g}$; this correspondence is given by
\begin{gather*}
 sX = J^1 X + \nabla X, \qquad X \in \Gamma({\mathfrak g}). 
\end{gather*}
(According to our sign conventions, the inclusion $T^*\!M \otimes {\mathfrak g} \rightarrow J^1 {\mathfrak g} $ induces the following map on sections: $df \otimes X \mapsto f J^1 X - J^1 (f X)$.) When $s \colon {\mathfrak g} \rightarrow J^1 {\mathfrak g} $ is a morphism of Lie algebroids, as is the case here, then $\nabla $ is called a {\it Cartan connection} on~$\mathfrak g$~\cite{Blaom_06}, or an {\it infinitesimal Cartan connection} when we want to distinguish it from the corresponding global notion def\/ined in Section~\ref{onepoint}. As usual, the {\it curvature} of $\nabla $ is def\/ined by
\begin{gather*}
 \curvature \nabla (X,Y)Z=\nabla_X \nabla_Y Z - \nabla_Y \nabla_X Z - \nabla_{[X,Y]}Z,
\end{gather*}
and $\nabla $ is {\it flat} if $\curvature \nabla =0$.

\subsection{The main theorem}\label{themain}
The central question to be addressed in the present article is this:
\begin{aside}
To what extent are obstructions to the involutivity of a Cartan connection $D$ encoded in the corresponding infinitesimal Cartan connection~$\nabla$?
\end{aside}
To answer this question, let $D$ be a Cartan connection on a Lie groupoid $G$ and let $G^D\subset G$ denote the set of points through which there exist $n$-dimensional integral manifolds of $D$, $n=\rank D=\dim M$.
\begin{Theorem}\label{theorem1.2}\quad
\begin{enumerate}\itemsep=0pt
\item[$(1)$]
If a Cartan connection $D$ on a Lie groupoid~$G$ over $M$ is involutive in some neighbor\-hood~$V$ of~$M$ then the corresponding infinitesimal Cartan connection $\nabla$ is flat.
\item[$(2)$] 
$G^D$ is a wide subgroupoid of $G$ and, when $\nabla $ is flat, a union of connected components of~$G$ containing the source-connected component $G^{0}$ of~$G$.
\end{enumerate}
\end{Theorem}

The proof of this theorem and its connection to other work is discussed in Sections~\ref{sal} and~\ref{mmm} below.

\subsection{Pseudoactions}\label{ergs}
If $\nabla $ above is f\/lat then, replacing $G$ with a union of connected components, $D$ is tangent to a~foliation ${\mathcal F}$ on~$G$. We now axiomatize the sense in which ${\mathcal F} $ is compatible with the multiplicative structure of~$G$, and how it accordingly generates a pseudogroup of local transformations on~$M$ (the pseudogroup of isometries of~$M$ in the case of a maximally symmetric Riemannian manifold).

Let $G$ be a Lie groupoid and call an immersed submanifold $\Sigma\subset G$ a {\it pseudotransformation} if the restrictions to $\Sigma$ of the groupoid's source and target maps are local dif\/feomorphisms. A~{\it pseudoaction}~\cite{Blaom_16} on $G$ is any smooth foliation ${\mathcal F}$ on $G$ such that:
\begin{enumerate}\itemsep=0pt
 \item[(1)] The leaves of ${\mathcal F}$ are pseudotransformations.
 \item[(2)] ${\mathcal F}$ is {\it multiplicatively closed}.
\end{enumerate}
To def\/ine what is meant in (2), regard local bisections of~$G$ as immersed submanifolds and let~$\widehat{\mathcal F}$ denote the collection of all local bisections that intersect each leaf of~${\mathcal F} $ in an {\em open} subset. Let~$ \widehat G$ denote the collection of {\em all} local bisections of~$G$, this being a groupoid over the collection of all open subsets of $M$. Then condition~(2) is the requirement that $\widehat{\mathcal F}\subset \widehat G$ be a subgroupoid.

Given a pseudoaction ${\mathcal F}$ of $G$ on $M$, each element $b \in \widehat{\mathcal F}$ def\/ines a local transformation~$\phi_b$ of~$M$. Let $\pseud({\mathcal F})$ denote the set of all local transformations $\psi \colon U \rightarrow V$ of $M$ that are locally of this form; that is, for every $m \in U$, there exists a~neighborhood $U' \subset U$ of $m$ such that $\psi|_{U'}= \phi_b$ for some for some $b \in\widehat {\mathcal F}$. Then, by~(2), $\pseud({\mathcal F})$ is a pseudogroup.

A pseudoaction generalizes the ordinary action of a Lie {\em group}~$G_0$ on $M$: Take $G$ to be the corresponding action groupoid $G_0 \times M$ and ${\mathcal F} $ the foliation whose leaves are $\{g\}\times M$, $g \in G_0$. For more examples see~\cite{Blaom_16}, where the following elementary observation is established:
\begin{Proposition}\label{proposition-ergs} An arbitrary foliation ${\mathcal F}$ on $G$ is a pseudoaction if and only if its tangent $n$-plane field $D$ is a Cartan connection on~$G$.
\end{Proposition}

Under an additional completeness hypothesis, a pseudoaction becomes an ordinary Lie group action, with possibly a monodromy `twist'~\cite{Blaom_16}.

An immediate corollary to Theorem~\ref{theorem1.2} can now be stated as follows:
\begin{Corollary}\label{corollary-ergs}
\looseness=-1 Assume $G$ is source-connected. Then a Cartan connection $D$ is tangent to a~pseudoaction ${\mathcal F} $ on~$G$ if and only if the corresponding infinitesimal Cartan connection $\nabla $ is flat.
\end{Corollary}

In practice, the pseudogroup generated by a pseudoaction ${\mathcal F}$ represents a pseudogroup of {\em symmetries}, the Riemannian example mentioned in Section~\ref{reg} being a case in point. We consequently make the following interpretation:
\begin{aside}
 A Cartan connection on a Lie groupoid is a $($possibly intransitive$)$ `symmetry deformed by curvature'.
\end{aside}

\subsection{Connections to other work}\label{sal}
Theorem \ref{theorem1.2}(1) and the above corollary already follow from a result of Salazar \cite[Theo\-rem~6.4.1]{Salazar_13} (see also~\cite{Crainic_etal_15}) for arbitrary `multiplicatively closed' distributions on a Lie groupoid, of which the $n$-plane f\/ield $D$ def\/ined by a Cartan connection is a special case. The independently obtained proof of Theorem~\ref{theorem1.2} given in Section~\ref{srt} is based on: (i) A detailed analysis of the multiplicative structure of~$J^1 G$, given in Section~\ref{sec3}; and (ii) A characterization of the inf\/initesimal connection $\nabla$ as inf\/initesimal parallel translation in the Lie groupoid $G$ along $D$, described in Section~\ref{gg}. We expect both analyses will be of some independent interest.

In Salazar's approach, one identif\/ies the obstruction to the involutivity of $D$ with a Lie groupoid cocycle $c$, whose inf\/initesimalization $\kappa$, a Lie algebroid cocycle, is seen to vanish when~$\nabla$ is f\/lat. Van Est's isomorphism between Lie groupoid and Lie algebroid cohomology, generalized to Lie groupoids and algebroids~\cite{Crainic_03}, then implies that the cohomology class $[c]$ must vanish. To show $c$ itself vanishes requires a~further topological argument.

Although the basic idea behind Salazar's proof and the present one are reasonably straightforward, f\/lushing out the details requires some work in both cases. Salazar's result has the advantage that it applies in greater generality; we hope that our proof sheds light on the special case important to the theory of Cartan connections.

As far as we know, the f\/irst statement (without proof) of Corollary~\ref{corollary-ergs} appears in \cite[Appendix~A]{Blaom_12}.

By Corollary \ref{corollary-ergs}, the inf\/initesimalization of a pseudoaction is a {\em flat} Cartan connection on a Lie algebroid, also called by us a {\it twisted Lie algebra action}, for reasons explained in~\cite{Blaom_16}. Analogues of the so-called fundamental theorems, Lie~I, Lie~II and Lie~III, for Lie groups and Lie algebras, are established for pseudoactions and twisted Lie algebra actions in~\cite{Blaom_16}, to which the present article can be regarded as a prequel.

The modern redevelopment of Cartan geometries along lines promoted here, and in our other works cited above, has several aspects warranting further investigation. We brief\/ly discuss two here.

First, there have yet to appear many interesting applications of the theory to {\em intransitive} geometric structures (structures that Ehresmann's formulation cannot handle). One reason for this may be the focus of current research on intransitive structures of {\em infinite} type, a subject far richer than its transitive counterpart. These include: Dirac structures~\cite{Courant_90}, which gene\-ra\-li\-ze Poisson, presymplectic, and foliated structures; genera\-li\-zed complex structures~\cite{Hitchin_03}; and Jacobi structures~\cite{Kirillov_76a,Kirillov_77, Lichnerowicz_78} which simultaneously generalize Poisson and contact structures (see also~\cite{Crainic_Salazar_15}). However, add extra structure, or perturb the underlying `integrability condition', and these structures generally become f\/inite-type structures which presumably admit a~Cartan geometry interpretation. Moreover, adding extra structure, such as a compatible metric, may shed light on the underlying inf\/inite-type structure, in the same way that equipping a Riemann surface with a~metric sheds light on the underlying topology of the surface.

Secondly, quite a large number of examples of transitive geometric structures~-- including conformal, CR, and projective structures~-- are of so-called {\it parabolic} type, a notion cast very decidedly in Ehresmann's formulation of a Cartan geometry (including model data)~\cite{Cap_Slovak_09}. It would be worthwhile f\/inding an axiomatization of parabolic geometries in terms of Lie groupoids or Lie algebroids. An intriguing observation in this direction is that the Atiyah algebroid of a~parabolic geometry is also a twisted Courant algebroid~\cite{Armstrong_07,Armstrong_Lu_12,XuXiaomeng_13}.

\subsection[Multiplication in $J^1 G$ and parallel translation]{Multiplication in $\boldsymbol{J^1 G}$ and parallel translation}\label{mmm}

\looseness=-1 In making a proof of Theorem \ref{theorem1.2}, we have been led to make a detailed investigation of the multiplicative structure of~$J^1G$, which is given in Section \ref{sec3}. While not hard to guess at this description~-- at the level of the bisection group, it is a~semidirect product~-- in writing out details the following technical problem arises: Given two paths in a Lie groupoid $g(t), h(t) \in G$ such that the product $g(t)h(t)$ is always def\/ined, compute the time derivative of the product $g(t)h(t)$ in terms of the time derivatives of the individual factors. In the category of Lie {\em groups}, a~trivial application of the chain rule solves the problem, but this application rests on the fact that $g(t)h(s)$ is def\/ined for $s\neq t$, which needn't be the case for arbitrary Lie groupoids. The solution adopted here is to parallel translate the factors using an appropriate connection~$D$ on the groupoid.

As we show in Section \ref{gg}, if one inf\/initesimalizes the parallel translation associated with a~{\em Cartan} connection $D$ on $G$, then
one recovers the corresponding inf\/initesimal Cartan connection~$\nabla$ on~${\mathfrak g} $ def\/ined in Section~\ref{hsd} above. This alternative
characterization of~$\nabla $ is the key to our proof of Theorem~\ref{theorem1.2}.
\vspace{-1mm}

\section{Classical Cartan connections} \label{ugt}

We now construct, for each classical Cartan connection $\omega $ on a~principal $H$-bundle $P$, a Cartan connection $S^\omega$ on the gauge groupoid $(P \times P)/H$. A converse construction, completing the proof of Theorem~\ref{theorem1.1}, is postponed until Section~\ref{recover}. In Section~\ref{foobar}
we describe the inf\/initesimalization of $S^\omega$, a Cartan connection $\nabla^\omega $ on the Atiyah Lie algebroid~$TP/H$ of~$P$.

In this section only, $G$ denotes a Lie {\em group}, and not a Lie groupoid.

\subsection{Classical Cartan connections}\label{classical}
Let $H$ be a Lie group with Lie algebra ${\mathfrak h}$, let $\pi \colon P \rightarrow M$ be a right principal $H$-bundle, and suppose~$P$ is equipped with a~parallelism $\omega $. Then $\omega $ is called a {\it Cartan connection} on $P$ if:
\begin{enumerate}\itemsep=-0.5pt
\item[(1)] The inf\/initesimal generators $\xi^\dagger$, $\xi \in {\mathfrak h}$, of $H$ are $\omega$-parallel.
\item[(2)] The space of $\omega$-parallel vector f\/ields is invariant under the canonical right action of $H$ on the space of all vector f\/ields (the action by pullback).
\end{enumerate}
In this article, to distinguish this use of the term from the one given previously, we will call $\omega $ a {\it classical Cartan connection}.

We may understand the parallelism $\omega $ as a one-form on $P$ taking values in some vector space $V $, and whose restriction to each tangent space is an isomorphism onto $V$. Since we may identify~$V$ with the space of $\omega $-parallel vector f\/ields (the ones on which~$\omega $ is constant) condition~(1) realizes~${\mathfrak h}$ as a subpace of $V$ and we obtain the tautology
\begin{enumerate}\itemsep=-0.5pt
\item[(3)] $\omega(\xi^\dagger(p))=\xi$, for all $\xi \in {\mathfrak h}$ and $p \in P$. 
\end{enumerate}
On the other hand, condition~(2) means that $V$ becomes a~right representation of $H$, and this action extends the adjoint action in the sense that $\xi h = \Adjoint_{h^{-1}} \xi $ whenever $\xi \in {\mathfrak h} $. Tautologically, then, we also have the equivariance condition
\begin{enumerate}\itemsep=0pt
\item[(4)] $\omega (vh)= \omega (v)h$ for all $h \in H$, $v \in TP$. 
\end{enumerate}
More commonly, one {\em prescribes} a vector space $V$, an inclusion ${\mathfrak h} \subset V$ and a right representation of $H$ on $V$ extending the adjoint action, and def\/ines a Cartan connection as a~parallelism $\omega \colon TP \rightarrow V$ satisfying~(3) and~(4). However, in this case it is easy to see that~(3) and (4) imply~(1) and~(2) and that consequently the two formulations are equivalent.

For example, if we take $H$ to be a closed subgroup of some Lie group $G$, acting on $G$ from the right, then we may take $V={\mathfrak g}$, where ${\mathfrak g} $ is the Lie algebra of $G $, and take $H$ to act on ${\mathfrak g} $ by the adjoint action: $\xi h := \Adjoint_{h^{-1}} \xi $, $\xi \in {\mathfrak g}$. Then the left-invariant Mauer--Cartan form $\omega $ on $G$ is a~Cartan connection on the principal bundle $G \rightarrow G/H$ satisfying the well-known Mauer--Cartan equation $d \omega - \frac{1}{2}\omega \wedge \omega = 0$, where a wedge implies a contraction using the bracket on~${\mathfrak g} $.\footnote{We def\/ine brackets on Lie algebras and Lie algebroids using {\em right}-invariant vector f\/ields.}

\subsection{Cartan connections with model data}
Generalizing the previous example, we may consider Cartan connections on an arbitrary principal $H$-bundle $P \rightarrow M$ for which $V={\mathfrak g} $ and $H$ acts via restricted adjoint action. One refers to $G/H$ as the {\it homogeneous model} and the extra data (the Lie group $G$ and the embedding $H \subset G$) as {\it model data}. With model data prescribed, one def\/ines the {\it curvature} of $\omega $ by $\Omega = d \omega - \frac{1}{2} \omega \wedge \omega
$, where the wedge again implies a contraction using the bracket on ${\mathfrak g}$ (part of the model data!). If $\Omega =0$, then~$M$ can be identif\/ied, at least locally, with the homogeneous model~$G/H$ \cite[Theorem~5.3]{Sharpe_97}. More generally, this leads one to regard $M$ with the given data as a homogeneous space `deformed by curvature'.

Two points must be emphasised when comparing this `deformed by curvature' interpretation with the analogous interpretation of Cartan connections on Lie groupoids: First, the homogeneous model $G/H$ is specif\/ied {\em a priori}; in particular, if $M$ happens to be a {\em different} homogeneous space, then generally $\Omega \ne 0$. Secondly, only {\em transitive} symmetry can follow from the `f\/latness' ($\Omega=0$) of a classical Cartan connection.

In fact, since every classical Cartan connection determines a Cartan connection on a Lie groupoid, there is already a notion of curvature which does not depend on the prescription of model data, but its interpretation is dif\/ferent: in the transitive case this curvature vanishes when~$M$ is locally isomorphic to {\em some} homogeneous space, not necessarily a model f\/ixed beforehand. See also~\cite{Blaom_13} and Remark~\ref{final-remark}. 

The literature devoted to classical Cartan connections which include model data is large. The reader is referred to the book~\cite{Sharpe_97} for an introduction, and to~\cite{Cap_Slovak_09} for a more recent survey. Note that an important work in which model data is {\em not} prescribed is Morimoto's construction of Cartan connections for f\/iltered manifolds~\cite{Morimoto_93}.

\subsection{Cartan connections from classical Cartan connections}\label{grr}
The notion of a parallelism on a smooth manifold $P$ can be reformulated using Lie groupoid language as follows: Let $S \colon P \times P \rightarrow J^1 (P \times P)$ be any Cartan connection on the pair groupoid $P \times P$ in the sense of Section~\ref{onepoint}. Note that elements of $J^1 (P \times P)$ may be understood as one-jets of local transformations $\phi \colon U \rightarrow V$ of $P$. Now $J^1 (P \times P)$, and hence the Lie subgroupoid $G=S(P \times P)$, acts canonically on the tangent bundle $TP$. Call a vector f\/ield on $P$ {\it parallel} if it is $G$-invariant, and denote the $n$-dimensional space of parallel vector f\/ields by $V,$ $n=\dim P$. The tautological map $\omega \colon TP \rightarrow V$ satisfying $\omega(X(m))=X$ for all $X \in V$ is a parallelism, in sense of Section~\ref{classical}, and evidently all parallelisms arise in this way.

Now suppose $\pi \colon P \rightarrow M$ is a right principal $H$-bundle and let $H$ act on $J^1 (P \times P)$ from the right by Lie groupoid morphisms according to
\begin{gather*}\big(J^1_p \phi\big)h =J^1_{ph}\big( R_h \circ \phi \circ R_{h^{-1}}\big),\end{gather*} where
$R_h(p):=ph$, and let $J^1 (P\times P)^{{\mathfrak h}} \subset J^1 (P \times P)$ denote the subset consisting of one-jets of local transformations $\phi $ of $P$ that are inf\/initesimally $H$-equivariant, which means $\phi^* \xi^\dagger = \xi^\dagger$, for all $\xi $ in ${\mathfrak h}$. Then $J^1(P \times P)^{\mathfrak h}$ is a transitive $H$-invariant Lie subgroupoid of $ J^1(P \times P)$, as is readily verif\/ied.
\begin{Proposition}
The parallelism $\omega \colon TP \rightarrow V$ defined by a Cartan connection $S$ on $P \times P$ is a classical Cartan connection on~$P$ if and only if~$S$ satisfies the following conditions:
 \begin{enumerate}\itemsep=0pt
 \item[$(1)$] $S \colon P \times P \rightarrow J^1 (P \times P)$ takes values in $J^1(P \times P)^{{\mathfrak h}} $.
 \item[$(2)$] ${S(qh,ph)}= S(p,q) h$ for $p,q \in P$ and $h \in H$ $(H$-equivariance$)$.
\end{enumerate}
\end{Proposition}

The easy proof is left to the reader.

The key observation that allows one to drop a Cartan connection $S$ on $P \times P $ to a Cartan connection on the gauge groupoid $(P \times
P)/H$ is the following:

\begin{Lemma}\label{lemma2.2}
 As Lie groupoids over $M$, we have $J^1((P \times P)/H)\cong J^1(P \times P)^{{\mathfrak h}}/H$.
\end{Lemma}

The following comments will help the reader follow the proof:
\begin{enumerate}\itemsep=0pt
\item[(3)] Our convention is to regard the source map of the pair groupoid $P \times P$ as $(q,p) \mapsto p$.
\item[(4)]
Given an element $J^1 _p \phi \in J^1(P \times P)^{\mathfrak h}$ represented by some inf\/initesimally $H$-equivariant $\phi $, the maps $R_h \circ \phi$ and $\phi \circ R_h$ have the same germ at $p$, for all $h \in H$ in some neighborhood of the identity. This is because inf\/initesimal
 $H$-equivariance implies local $H$-equivariance. (Consequently, for all $H $ suf\/f\/iciently close to the identity, we simply have $(J^1_p \phi)h = J^1 _{ph} \phi $.)
\item[(5)]
Viewing a bisection of $(P \times P)/H$ as a right inverse for the source projection $(q,p)\modulo H \mapsto \pi(p)$, the local triviality of the bundle $\pi \colon P \rightarrow M$ allows us to write such a bisection in the form $b(m)=(q(m),p(m))\modulo H$, for some smooth local right-inverse \mbox{$m \mapsto p(m)$} of $\pi \colon P \rightarrow M $, and some smooth map $m \mapsto q(m)$, where the composite $\pi \circ q$ is a~dif\/feo\-mor\-phism onto its image.
\end{enumerate}
\begin{proof}[Proof of Lemma~\ref{lemma2.2}]
 A map $A \colon J^1(P \times P)^{{\mathfrak h}}/H \rightarrow J^1((P \times P)/H) $ is def\/ined as follows. Given an element $J^1_{p_0} \phi \modulo H \in J^1(P \times P)^{\mathfrak h} / H$ represented by some inf\/initesimally $H$-equivariant local transformation $\phi $ of~$P$, f\/irst choose a local right-inverse $m \mapsto p(m)$ for $\pi \colon P \rightarrow M$ satisfying $p(m_0)=p_0$. Then we put
 \begin{gather*} A\big(J^1_{p_0} \phi \modulo H\big) = J^1_{m_0}b,\end{gather*}
 where $m_0=\pi(p_0)$ and $b$ is the local bisection of $(P \times P) /H$ def\/ined by
 \begin{gather*}
 b(m)=\big(\phi(p(m)), p(m)\big) \modulo H.
 \end{gather*}

Let $m \mapsto \bar p(m)$ be a second local right-inverse for $\pi$ satisfying $\bar p(m_0)=p_0$. Then $\bar p(m)=p(m)h(m)$, for some smooth map $m \mapsto h(m) \in H$ satisfying $h(m_0)=\identity$. For all $m$ suf\/f\/iciently close to $m_0$ we have $\phi(p(m)h(m))=\phi(p(m))h(m)$, by~(4), in which case $(\phi(\bar p (m)), \bar p(m)) \modulo H $ $= (\phi(p (m)), p(m)) \modulo H$, which shows that the def\/inition of $b$ is (up to germ equivalence at~$m_0$) independent of the choice of $m \mapsto p(m)$.

To show that the def\/inition of $A(J^1_{p_0} \phi \modulo H)$ is independent of the choice of representative~$J^1_{p_0} \phi$, it is enough to observe: (i)~if $J^1_{\bar p_0} \bar \phi $ is a second representative then the one-jet of~$\bar \phi $ at~$\bar p_0$ coincides with the one-jet of $R_h \circ \phi \circ R_{h^{-1}}$, where $h \in H$ is determined by $\bar p_0=p_0h$; and (ii)~in that case a local right-inverse $m \mapsto \bar p(m)$ for~$\pi$ satisfying $\bar p (m_0) = \bar p_0$ is given by $\bar p(m)=p(m)h$, where~$p(m)$ is as above.

We leave it to the reader to show that $A$ is a Lie groupoid morphism. To show $A$ is injective, let $J^1_{m_0}b \modulo H$ be an element of $J^1 ((P \times P) /H)$ with $b$ as in~(5). Since each point of $P$ in a neighborhood of $p_0:=p(m_0)$ is of the form~$p(m)h$, for some uniquely determined $h \in H$ and \mbox{$m \in M$}, a local transformation $\phi $ of $P$ is well-def\/ined by $\phi(p(m)h)=q(m)h$. Then $\phi $ is $H$-equivariant by construction, and hence inf\/initesimally $H$-equivariant, and we have $A(J^1_{p_0}\phi \modulo H)=J^1_{m_0}b$.

The proof that $A$ is injective is no more dif\/f\/icult and is omitted.
\end{proof}

Now let $\omega $ be a classical Cartan connection, understood as the parallelism def\/ined by a Cartan connection $S$ on $P \times P$ satisfying~(1) and~(2). By the lemma, a Cartan connection $S^\omega$ on the quotient $(P \times P)/H $ is the same thing as a Lie groupoid morphism
\begin{gather*}S^\omega \colon \ (P \times P)/H \rightarrow J^1(P \times P)^{\mathfrak h} / H\end{gather*}
furnishing a right-inverse for the canonical projection $J^1(P \times P)^{\mathfrak h} / H \rightarrow (P \times P)/H$. By the preceding proposition, such a~connection is well-def\/ined by
\begin{gather*}S^\omega\big((p,q) \modulo H\big) = S(p,q) \modulo H, \qquad p,q \in P.\end{gather*}

\subsection{Inf\/initesimalization of a classical Cartan connection}\label{foobar}
Let $\omega \colon P \rightarrow V$ be a classical Cartan connection on a principal $H$-bundle as above. Then the inf\/initesimalization of the Cartan connection $S^\omega $ is a Cartan connection $\nabla^\omega$ on the Atiyah Lie algebroid $TP/H$. To show how $\nabla^\omega$ is related to $\omega$, let $\bar\nabla $ denote the inf\/initesimal parallelism determined by $\omega $, i.e., the f\/lat Koszul connection on $TP$ for which the notions of $\omega $-parallel and $\bar\nabla$-parallel coincide; $\bar\nabla $ is characterised by the identity
\begin{gather}
 \omega(\bar\nabla_XY)=d (\omega(Y))(X), \qquad X,Y \in \Gamma(TP).\label{lisk}
\end{gather}
\begin{Lemma} Let $X$ and $Y$ be $H$-invariant vector fields on $P$. Then:
 \begin{enumerate}\itemsep=0pt
 \item[$(1)$] $\bar\nabla_XY$ is $H$-invariant.
 \item[$(2)$] $\bar\nabla_XY=[X,Y]$ whenever $Y$ is vertical.
 \end{enumerate}
\end{Lemma}

\begin{proof}
Since $\omega $ is $H$-equivariant, a vector f\/ield on $P$ is $H$-invariant if and only if the $V$-valued function $\omega(X)$ is $H$-equivariant. Given this, and using the fact that the time-$t$ f\/low map of an $H$-invariant vector f\/ield commutes with the action of~$H$, it is not hard to prove~(1). To prove~(2)
 it suf\/f\/ices, by \eqref{lisk}, to prove that the Lie derivative ${\mathcal L}_X \omega$ vanishes on vertical vectors. To this end, let $\xi \in {\mathfrak h} $ and $p \in P$ be given. Then, denoting the time-$t$ f\/low map of $X$ by $\Phi_X^t$, we compute{\samepage
 \begin{align*}
 ({\mathcal L}_X \omega)(\xi^\dagger(p))&=\frac{d}{dt} \omega \left( T \Phi_X^t \cdot \frac{d}{ds}p\exp(s \xi)\Big|_{s=0}\right)\Big|_{t=0}\\
 &=\frac{d}{dt} \omega \left(\frac{d}{ds} \Phi_X^t(p\exp(s \xi))\Big|_{s=0}\right)\Big|_{t=0}\\
 &=\frac{d}{dt} \omega \left(\frac{d}{ds} \Phi_X^t(p) \exp(s \xi)\Big|_{s=0}\right)\Big|_{t=0}\\
 &=\frac{d}{dt} \omega\big(\xi^\dagger\big(\Phi_X^t(p)\big)\big)\Big|_{t=0}=\frac{d}{dt}\xi \Big|_{t=0}=0.
 \end{align*}
 The third equality holds because the time-$t$ f\/low map of $X$ commutes with the action of $H$.}
\end{proof}

Now the space of $H$-invariant vector f\/ields on $P$ is in one-to-one correspondence with the sections of the Atiyah Lie algebroid $TP/H$. By part~(1) of the lemma, the f\/lat $TP$-connection~$\bar\nabla $ therefore drops to a representation $\bar\nabla^\omega$ of the Lie algebroid $TP/H$ on itself. According to part~(2) of the lemma, we have, for sections $X$, $Y$ of $TP/H$
\begin{gather*}
 \bar\nabla^\omega_XY=[X,Y],
\end{gather*}
whenever $\#Y=0$. Here $\#$ denotes the anchor of $TP/H$ and $[\cdot,\cdot]$ the canonical bracket on its sections. It follows that there is a $TM$-connection $\nabla^\omega$ on $TP/H$ def\/ined implicitly by
\begin{gather*}
 \bar\nabla_X^\omega Y = \nabla^\omega_{\#Y}X +[X,Y].
\end{gather*}
The proof of the following is left as an exercise for the reader:
\begin{Proposition}
 The connection $\nabla^\omega $ is the infinitesimalization of the Cartan connection $S^\omega$.
\end{Proposition}

In particular, $\nabla^\omega $ is an inf\/initesimal Cartan connection, a fact already observed in~\cite{Blaom_06}.

\begin{Remark}\label{final-remark}
As a f\/inal remark, we observe that if we prescribe model data and def\/ine $\Omega =d \omega - \frac{1}{2}\omega \wedge \omega $, then $\nabla^\omega $ is f\/lat~-- implying $S^\omega$ def\/ines a~pseudoaction on a union of connected components of $(P \times P)/H$~-- precisely when $\bar\nabla \Omega =0$ \cite[Theorem~C]{Blaom_06}, a~condition weaker than the vanishing of $\Omega$.
\end{Remark}

\section[The multiplicative structure of $J^1 G$]{The multiplicative structure of $\boldsymbol{J^1 G}$}\label{sec3}
Let $G$ be a Lie groupoid. The Lie algebroid of $J^1 G$ can be identif\/ied with $J^1 {\mathfrak g} $, where ${\mathfrak g} $ denotes the Lie algebroid of~$G$ (see Section~\ref{repsagain} below). The canonical exact sequence,
\begin{gather}
 0\longrightarrow T^*\!M\otimes{\mathfrak g}\longrightarrow J^1{\mathfrak g}\longrightarrow{\mathfrak g}\longrightarrow 0\label{exact}
\end{gather}
may be regarded as the derivative of a natural sequence of groupoid morphisms,
\begin{gather*}
 \model\monomorphism J^1 G \rightarrow G.
\end{gather*}
Here $\model$ is a certain open neighborhood of the zero-section of $T^*\!M \otimes {\mathfrak g}$, def\/ined in Section~\ref{mod} below; if $G=M\times M$, then $\model=\automorphism({TM})$.

Corresponding to \eqref{exact} is an exact sequence of section spaces which splits canonically, leading to an identif\/ication,
\begin{gather*}
 \Gamma\big(J^1{\mathfrak g}\big)\cong \Gamma({\mathfrak g})\oplus \Gamma(T^*\!M \otimes {\mathfrak g}).
\end{gather*}
Under this identif\/ication, the Lie bracket on $\Gamma(J^1{\mathfrak g})$ is a semidirect product~\cite{Blaom_06}.

In this section we establish the global analogue of this result, namely a semidirect product structure,
\begin{gather*}
 {B}\big(J^1 G\big)\cong {B}(G)\times {B} (\model),
\end{gather*}
where ${B}(\, \cdot\,)$ denotes the group of global bisections.

Just as a choice of Cartan connection $\nabla $ on $\mathfrak g$ determines an identif\/ication
\begin{gather*}J^1{\mathfrak g}\cong {\mathfrak g} \oplus (T^*\!M \otimes {\mathfrak g}) \end{gather*}
and an associated semi-direct product structure for $J^1{\mathfrak g} $, so a choice of Cartan connection on~$G$ determines a semi-direct product structure for~$J^1 G$. While we shall provide a direct demonstration of this fact, the reader may like to interpret the existence of the semi-direct product structure as a~consequence of the fact that $J^1 G \rightarrow G$ is Lie groupoid f\/ibration, and $S$ a unital f\/lat cleavage for it \cite[Theorem~2.5.3]{Mackenzie_05} (see also \cite[Theorem~2.2.3]{del_Hoyo_Fernandes_16}). In particular, $S$ determines a~representation of $G$ on the kernel $J^1_{M}G$ which def\/ines the semi-direct product structure.

The non-trivial part of our proof amounts to f\/inding an explicit formula for the representation of $G$ on $J^1_{M}G$ using the adjoint representation of $J^1 G $ on ${\mathfrak g} $ (see Section~\ref{adj} below) and our concrete model $\model$ of $J^1_{M}G$. This description rests on a detailed analysis of multiplication in~$J^1 G$. To formulate our claims in detail, which we do in Section~\ref{mainman}, will require some preparation.

\subsection[Representing elements of $J^1 G$ and its Lie algebra]{Representing elements of $\boldsymbol{J^1 G}$ and its Lie algebra}\label{repsagain}
There are three useful representations of an element of $J^1 G$. Formally, an element of $J^1 G$ is a~one-jet at some $m \in M$ of a~local bisection ${b} \colon U \rightarrow G$ of $G$. (It is our convention to regard local bisections as locally def\/ined right-inverses for the source projection $\alpha \colon G \rightarrow M$.) If $g={b}(m)$, then the tangent map $\mu:=T_m {b} $ is a~linear map from $T_mM$ to $T_gG$ satisfying:
\begin{enumerate}\itemsep=0pt
 \item[(1)] 
 $T_m \alpha \circ \mu = \identity_{{T_mM}}$, and
 \item[(2)] 
 $T_m \beta \circ \mu\colon T_mM \rightarrow T_mM$ is invertible.
\end{enumerate}
Conversely, for any $g \in G$ with $\alpha(g)=m$, any linear map $\mu \colon {T_mM} \rightarrow T_gG$ satisfying~(1) and~(2), is the tangent map of some local bisection ${b} \colon U \rightarrow G$ whose one-jet at $m \in M$ is independent of the particular choice of ${b}$. We are therefore entitled to identity elements of $J^1 G$ with linear maps ${T_mM} \rightarrow T_gG$ with $\alpha(g)=m$ and satisfying (1) and (2), {\em and will do so by default in the sequel.}

Finally, each $\mu \colon {T_mM} \rightarrow T_gG $ as above may be identif\/ied with its image, a subspace of $T_gG$ that is simultaneously a~complement for the tangent space at $g \in G$ to the f\/ibre of the source projection $\alpha \colon G \rightarrow M$ through $g$, and the f\/ibre of the target projection $\beta \colon G \rightarrow M$ through~$g$; and all such `simultaneous complements' may be realised in this way.

Let ${\mathfrak g} := {\mathcal L} (G)$ denote the abstract Lie algebroid of $G$. Our convention is to regard an element of~${\mathfrak g} $ as a vector tangent at some $m \in M \subset G$ to a~f\/ibre of the source projection $\alpha \colon G \rightarrow M$. Let~$J^1 {\mathfrak g} $ denote the vector bundle of one-jets of sections of~${\mathfrak g} $. Then there is a natural isomorphism $\theta \colon {\mathcal L}(J^1 G) \rightarrow J^1 {\mathfrak g} $ whose inverse is described as follows: Let $J^1 _mX\in J^1 {\mathfrak g} $ be given. Then $\theta^{-1}(J^1_mX)$ will be a vector tangent to $J^1 G$ at the identity element
$T_m \iota_M \colon {T_mM} \rightarrow T_{m}M$ of $J^1 G$ (with $\iota_M \colon M \rightarrow G$ denoting the inclusion) and so may be viewed as an equivalence class of paths in $J^1 G$ passing through $T_m \iota_M $. Specif\/ically,
\begin{gather}
 \theta^{-1}\big(J^1_m X\big)=\frac{d}{dt}T_m\big(\Phi_{X^\mathrm{R}}^t\circ \iota_M\big)\Big|_{t=0}, \label{sop}
\end{gather}
where $X^\mathrm{R}$ denotes the corresponding right-invariant vector f\/ield on $G$ and $t \mapsto \Phi_{X^\mathrm{R}}^t$ its f\/low.

\subsection{The adjoint representation}\label{adj}
The adjoint representation of a Lie group $G$ is a representation of $G$ on its Lie algebra. More generally, for each Lie groupoid~$G$, one has a god-given representation of $J^1 G$ on the Lie algebroid~${\mathfrak g}$ of~$G$, which we also call the {\it adjoint representation}; it is def\/ined as follows: Let $\mu \in J_g^1G$ be given, where $g \in G$ begins at $m \in M$ and ends at $m' \in M$, and let $X \in {\mathfrak g}|_{m}$. First, notice that we may use $\mu $
to lift $X$, viewed as a vector tangent to the source-f\/ibre at $m$, to a vector $X'$ tangent to the same source-f\/ibre at $g$:
\begin{gather*}
 X' = \mu(\#X) - T L_g \cdot T I\cdot X.
\end{gather*}
Here $\#$ denotes the anchor, $I \colon G \rightarrow G$ inversion and $L_g \colon G \rightarrow G$ left-multiplication by $g$. (One has, incidentally, $T I \cdot X = \# X - X$, for all $X \in {\mathfrak g}$.) Then we def\/ine
\begin{gather*}
 \Adjoint_\mu X := T R_{g^{-1}} \cdot X' = T R_{g^{-1}} \cdot \big( \mu(\#X) - T L_g \cdot T I \cdot X\big),
\end{gather*}
where $R_g$ is right-multiplication by $g$. Notice that if $\# X =0$ (always true if $G$ is a Lie {\em group}) then this collapses to the familiar
\begin{gather*}
 \Adjoint_\mu X = T \big(R_{g^{-1}} \circ L_g\big) \cdot X.
\end{gather*}

Additionally, there is a natural representation of $J^1 G$ on $TM$, with respect to which the anchor $\# \colon {\mathfrak g}\rightarrow TM$ is equivariant; this action, also denoted $\Adjoint$, is given by
\begin{gather*}
 \Adjoint_\mu v := T \beta \cdot \mu(v), \qquad v \in T_m M,
\end{gather*}
where $\beta \colon G \rightarrow M$ is the target projection.

Finally, the two representations just def\/ined induce a representation of $J^1 G$ on $T^*\!M \otimes {\mathfrak g}$ needed later:
\begin{gather*}
 (\Adjoint_\mu \phi)v = \Adjoint_\mu \big(\phi \Adjoint_\mu^{-1} v\big), \qquad \phi \in (T^*\!M \otimes {\mathfrak g})|_m.
\end{gather*}

For the sake of completeness, we add the following observation:
\begin{Remark}
 The adjoint representation amounts to a Lie groupoid morphism
 \begin{gather*}\Adjoint \colon \ J^1 G \rightarrow \operatorname{GL}({\mathfrak g}),\end{gather*}
where $\operatorname{GL}({\mathfrak g})$ is the frame groupoid of ${\mathfrak g} $. Inf\/initesimalizing, we obtain a Lie algebroid morphism $\adjoint \colon J^1 {\mathfrak g} \rightarrow \mathfrak{gl}({\mathfrak g}) $ coinciding with the adjoint representation of~$J^1 {\mathfrak g} $ on~${\mathfrak g} $, def\/ined by
 \begin{gather*} \adjoint_{J^1 X}Y=[X,Y].\end{gather*}
See, e.g.,~\cite{Blaom_06}. Here we are identifying the abstract Lie algebroid $\mathfrak{gl}({\mathfrak g})$ of $\operatorname{GL}({\mathfrak g})$ with the Lie algebroid of derivations on ${\mathfrak g} $ according to the following convention: Given an element of the abstract Lie algebroid represented by a path $t \mapsto \phi_t \in \operatorname{GL}({\mathfrak g}) $ with~$\phi_0$ the identity on~${\mathfrak g}|_m$ and~$\phi_t$ an isomorphism from ${\mathfrak g}|_m$ to ${\mathfrak g}|_{m_t}$, for some path $t \mapsto m_t \in M$, the corresponding derivation~$\partial$ is def\/ined by
\begin{gather*}
\partial Y = \frac{d}{dt}\phi_t^{-1} Y(m_t)\Big|_{t=0}, \qquad Y \in \Gamma({\mathfrak g}).
\end{gather*}
\end{Remark}

\subsection{Recovering a classical Cartan connection}\label{recover}
We pause our development to show how a Cartan connection $S \colon G \rightarrow J^1 G$ on a {\em transitive} Lie groupoid $G$ determines a~classical Cartan connection $\omega$ on the source-f\/ibre $P$ over an arbitrary f\/ixed point $m_0 \in M$. Recall that $P$ is a left principal $H$-bundle, where $H$ is the group of arrows of $G$ simultaneously beginning at terminating at $m_0$.

From the Cartan connection $S$ we obtain a representation of $G$ on its Lie algebroid ${\mathfrak g} $: $g \xi = \Adjoint_{S(g)}\xi$. This is indeed a representation because $S \colon G \rightarrow J^1 G$ is a Lie groupoid morphism. Letting $V = {\mathfrak g}|_{m_0}$, we def\/ine $\omega \colon TP \rightarrow V$ by $\omega(v)=g^{-1} (T{R}_{g^{-1}}\cdot v)$, where $g \in P$ is the base-point of $v \in TP$. One readily verif\/ies axioms~(3) and~(4) of Section~\ref{classical}
def\/ining a~classical Cartan connection. Under the identif\/ication $G \cong (P \times P)/H$ the Cartan connection $S$ coincides with~$S^\omega$, as def\/ined in Section~\ref{grr}, as the reader is invited to check. This completes the proof of Theorem~\ref{theorem1.1}.

\subsection[A linear model of the kernel of $J^1 G \rightarrow G$]{A linear model of the kernel of $\boldsymbol{J^1 G \rightarrow G}$}\label{mod}
Let $J^1_M G \subset J^1 G$ denote the pre-image of $M \subset G$ under the natural projection $J^1 G \rightarrow G$; $J^1 _M G$ is a totally intransitive Lie groupoid. We now def\/ine a certain open neighborhood of the zero-section of $T^*\!M \otimes {\mathfrak g}$, denoted $\model$, that serves as a model of $J^1_MG$.

For each $\phi \in T^*\!M \otimes {\mathfrak g} $ with base-point $m \in M$, def\/ine $\phi^{TM} \in T^*\!M \otimes {TM}$ by
\begin{gather*}
 \phi^{TM} v = v - \# \phi v.
\end{gather*}
Then, writing $\phi^{TM} \in \automorphism({TM})$ if $\phi^{TM} \colon {T_mM} \rightarrow T_mM$ is invertible, the set
\begin{gather*}
 \model := \big\{\phi \in T^*\!M \otimes {\mathfrak g} \,|\, \phi^{TM} \in \automorphism({TM})\big\}
\end{gather*}
is an open neighborhood of the zero-section of $T^*\!M \otimes {\mathfrak g}$. Moreover, $\model$ becomes a totally intransitive Lie groupoid over $M$, with the zero-section becoming the set of identity elements, if we def\/ine multiplication by
\begin{gather}
 \psi \phi = \psi + \phi - \psi \circ \# \circ \phi = \phi + \psi \circ \phi^{TM},\label{clost}
\end{gather}
and inversion by
\begin{gather*}
 \phi \mapsto -\phi \circ \big(\phi^{TM}\big)^{-1}.
\end{gather*}
With respect to this extra structure, the map
\begin{gather*}
\phi \mapsto \phi^{TM} \colon \ \model \rightarrow \automorphism({TM})
\end{gather*}
is a groupoid morphism (in this case, just a f\/ibre-wise group homomorphism) with commutative kernel. In other words, for each $m \in M$, $\model|_m$ is simply a~commutative extension of the image of $\phi \mapsto \phi^{TM}$, a~subgroup of~$\automorphism({TM})|_m$.

Finally, note that in addition to the natural representation of $\model $ on $TM$ given by $\phi \cdot v = \phi^{TM} v$, we have representation of $\model$ on ${\mathfrak g} $ given by
\begin{gather*}
 \phi \cdot X = \phi^{\mathfrak g} X := X - \phi\# X, \qquad X \in {\mathfrak g}.
\end{gather*}
With respect to these actions of $\model$ on ${\mathfrak g} $ and~$TM$, the anchor $\#$ is equivariant:
\begin{gather*}
 \#\phi^{\mathfrak g} X = \phi^{{TM}}\# X.
\end{gather*}
According to Theorem~\ref{theorem-mainman}(1) below, we have $\model \cong J^1_M G$ as Lie groupoids.

\subsection[Multiplication and inversion in $J^1 G$]{Multiplication and inversion in $\boldsymbol{J^1 G}$}\label{mainman}

Inversion in $J^1 G$ is straight-forward to describe. We leave the proof of the following to the reader:
\begin{Proposition}[inversion in $J^1 G$]
Suppose $g \in G$ begins at $m \in M$ and ends at $m' \in M$. Then, for any $\mu \in J_g^1G$ the inverse $\mu^{-1} \in J^1_{g^{-1}}G$ is given, as a map $T_{m'}M \rightarrow T_gG$, by
 \begin{gather*}
 \mu^{-1}(v) = TI \cdot \mu\big(\Adjoint_\mu^{-1} v\big), \qquad v \in T_{m'}M.
 \end{gather*}
 Here $I \colon G \rightarrow G$ denote inversion in $G$, and ${TI} \colon TG \rightarrow TG$ its tangent map.
\end{Proposition}

We now state our results for multiplication. First, recall that the Lie groupoid $\model$ is totally intransitive; its bisections coincide with its sections, as an open neighborhood of the zero-section of the vector bundle $T^*\!M \otimes {\mathfrak g}$. We have seen that the adjoint representation induces an action of $ J^1 G$ on $T^*\!M \otimes {\mathfrak g} $. It is not hard to show that $\model \subset T^*\!M \otimes {\mathfrak g} $ is invariant under this action, and we obtain an action of the bisection group ${B}(G)$ on~${B}(\model)$ def\/ined by
\begin{gather*}
 ({b} \cdot\Phi)(m') = \Adjoint_{T_m{b}}\Phi(m),
\end{gather*}
where $\beta(b(m))=m'$ def\/ines $m \in M$, and
$ {b} \in {B}(G) ,\Phi \in {B}(\model), m \in M$ are arbitrary. We
may therefore form the semidirect product of groups,
\begin{gather*}
 {B}(G)\times_{\Adjoint}{B}(\model),
\end{gather*}
which has multiplication def\/ined by
\begin{gather*}
 ({b}_1, \Phi_1)({b}_2,\Phi_2)=({b}_1
 {b}_2,\Phi_1 {b}_1\! \cdot\! \Phi_2).
\end{gather*}
According to Theorem~\ref{theorem-mainman}(2) below, this semidirect product is isomorphic to $B(J^1 G)$.

Secondly, suppose $G$ admits a Cartan connection $S \colon G \rightarrow J^1 G$. Then $G$ acts on $\model$ according to $g \cdot \phi := S(g) \cdot \phi =\Adjoint_{S(g)}\phi$. With this action in hand, the pullback $G \times_\beta \model$ of the f\/ibre-bundle $\model \rightarrow M$ under the target projection
$\beta \colon G \rightarrow M$ becomes a semi-direct product of Lie groupoids: The source and target projections are respectively $(g, \phi) \mapsto \alpha(g)$ and $(g, \phi) \mapsto \beta(g)$, where $\alpha, \beta \colon G \rightarrow M$ are the corresponding projections for $G$; multiplication is def\/ined by
\begin{gather*}
 (g_1, \phi_1)(g_2, \phi_2):=(g_1 g_2, \phi_1 \Adjoint_{S(g_1)} \phi_2).
\end{gather*}
According to Theorem~\ref{theorem-mainman}(3) below, $J^1 G \cong G \times_\beta \model$ as Lie groupoids.

\begin{Theorem}\label{theorem-mainman}For an arbitrary Lie groupoid $G$:
 \begin{enumerate}\itemsep=0pt
 \item[$(1)$]
The map $\phi \mapsto \varcheck \phi \colon \model \rightarrow J^1 G$, defined by $\varcheck \phi v = v - \phi v$, is a Lie groupoid morphism and embedding, whose image is the kernel of the natural projection \mbox{$J^1 G \rightarrow G$}. This map is equivariant with respect to the natural representations of both groupoids on~$TM$, and on~${\mathfrak g}$:
 \begin{gather*}
 \Adjoint_{\varcheck\phi} v = \phi^{TM}v,\qquad \Adjoint_{\varcheck\phi} X = \phi^{ \mathfrak g} X,\qquad v \in {TM},\qquad X \in {\mathfrak g}.
 \end{gather*}
 \item[$(2)$]
The map $a \colon {B}(G)\times_{\Adjoint}{B}(\model) \rightarrow {B}(J^1 G)$ defined by
 \begin{gather*}
 a({b},\Phi)(m)=(\Phi(m'))^\vee T_m {b}, \qquad m'=\beta({b}(m)),
 \end{gather*}
 is an isomorphism of Lie groups.
 \end{enumerate}
 Here $T_m {b} \colon {T_mM} \rightarrow T_{{b}(m)}G$ is the tangent map of ${b} $ at $m$, viewed as an element of $J^1G$, and $\beta \colon G \rightarrow M$ is the target projection. Also, $(\Phi(m'))^\vee T_m {b}$ denotes the product of $(\Phi (m'))^\vee $ and~$T_m b$ in the Lie groupoid $J^1 G$.
 \begin{enumerate}\itemsep=0pt
 \item[$(3)$]
 If $S \colon G \rightarrow J^1 G$ is a Cartan connection on $G$, then the map $c \colon G \times_\beta \model \rightarrow J^1 G$ defined by $c(g,\phi)= \phi^\vee S(g)$ is an isomorphism of Lie groupoids.
\end{enumerate}
\end{Theorem}

\subsection{Source connections and parallel actions}\label{source}
Before proving the preceding results, we introduce some terminology and notation we need when calculating the derivatives of paths in a~groupoid $G$ of the form $g(t)h(t)$ (and later when we establish a~geometric interpretation of an inf\/initesimalized Cartan connection.)

By a {\it source connection} on $G$, we shall mean a rank-$n$ subbundle $D \subset TG$ such that: (i)~$TG = D + \kernel T \alpha$; and (ii) $D(m)=T_mM$ for all $m \in M \subset G$. Here $n:=\dim(M)$ and $\alpha \colon G \rightarrow M$ is the source projection. A Cartan connection $D$, as def\/ined in Section~\ref{onepoint} above, is a~source connection. While Cartan connections need not exist in general, one can always f\/ind a~source connection locally (which will suf\/f\/ice for present purposes) and even globally if we work in the~$C^\infty$ category.

Suppose $a<0<b$. Fixing a source connection $D$ on $G$, and a smooth path $\gamma \colon (a,b) \rightarrow M$ on the base~$M$, we can def\/ine a local form of parallel translation along~$\gamma $. This parallel translation will be expressed in the language of Lie groupoid actions; see, e.g., \cite[Section~1.6]{Mackenzie_05}.

To simplify our description, we assume $\gamma $ is regular and simple, so that the pullback $\gamma^* G$ of the surjective submersion $\alpha \colon G\rightarrow M$ along $\gamma \colon (a,b) \rightarrow M$ can be identif\/ied with a submanifold of $G$, namely the pre-image of $\gamma((a,b))$ under $\alpha $. This pre-image is a union of source-f\/ibres, $P_t=\alpha^{-1}(\gamma(t))$, $a<t<b$. For any $U \subset \gamma^* G$ we write $(a,b)^2 \times_\gamma U$ for the set of all $(t_1,t_2, g) \in (a,b)^2 \times U$ such that $\alpha(g)=\gamma(t_1)$.

We claim that for any point $g_0 \in P_0 \subset \gamma^*G$ (or, more generally, any f\/inite number of points in $P_0$), there exists~-- shrinking the interval $(a,b)$ if necessary~-- an open neighborhood $U$ of $g_0$ in $\gamma^*G$, and an action of the pair groupoid $(a,b)^2$ on the restriction $\alpha \colon U \rightarrow \gamma((a,b))$, denoted $(t_1,t_2,g) \mapsto A^\gamma_{t_1,t_2}(g) \colon (a,b)^2 \times_\gamma U \rightarrow U$, such that the path $t
\mapsto A^\gamma _{t_0,t}(g)$ is $D$-horizontal, for any $t_0 \in (a,b)$ and $g \in P_{t_0}$. Indeed, $t \mapsto A^\gamma _{t_0,t}(g)$
is then the $D$-horizontal lift through $\alpha \colon G \rightarrow M$, of the path $t \mapsto \gamma(t)$, that passes through $g$ at time
$t_0$. The existence of the action follows, for example, from the observation that the connection $D$ on $\alpha \colon G \rightarrow M$
pulls back to a rank-one connection on the pullback $\gamma ^*G \rightarrow (a,b)$; f\/inding horizontal paths then amounts to integrating a vector f\/ield, to which the standard existence and uniqueness theory for f\/lows applies.

We will call the action $A^{\gamma}$ def\/ined above the {\it parallel action along $\gamma $} associated with the connection $D$. Note that $A_{t_1,t_2}^\gamma (m)=m$ because $D(m)=T_mM$ for all $m \in M$.

\subsection[Multiplication in $J^1 G$ in special cases]{Multiplication in $\boldsymbol{J^1 G}$ in special cases}\label{hotel}

The technical part of the proof of Theorem~\ref{theorem-mainman} is to derive product formulas in a few special cases. Referring to Fig.~\ref{fig1} below, we have:
\begin{Lemma}\label{lemma-hotel}
 Suppose $g \in G$ begins at $m \in M$ and ends at $m' \in M$. Then, for all $\mu,\nu \in J^1_gG$ and $\phi \in \model|_m$:
 \begin{enumerate}\itemsep=0pt
 \item[$(1)$] $\mu \varcheck\phi (v)=\mu(\phi^{TM} v)+ T L_g \cdot T I \cdot (\phi v)$, where $v \in T_mM$ is arbitrary;
 \item[$(2)$] $\nu \mu^{-1} = \varcheck \psi$, where
 \begin{gather*}
 \psi v=T R_{g^{-1}}\cdot \big(\mu\big(\Adjoint_\mu ^{-1}v\big)- \nu\big(\Adjoint_\mu^{-1} v\big)\big), \qquad v \in T_{m'}M;
 \end{gather*}
 \item[$(3)$] $\mu \varcheck\phi \mu^{-1} = (\Adjoint_\mu \phi)^\vee$.
 \end{enumerate}
 Also, if $\mu$, $\nu$ and $\phi$ above are related by $\nu= \mu\varcheck \phi $, then:
 \begin{enumerate}\itemsep=0pt
 \item[$(4)$] 
 $\mu(v)-\nu(v)=TR_g \cdot \Adjoint_\mu (\phi v)$, $v \in T_mM$.
 \end{enumerate}
Here $T I, T L_g, T R_{g^{-1}} \colon TG \rightarrow TG $ denote the tangent maps for inversion, left-multiplication by~$g$, and right-multiplication by~$g^{-1}$. The map $\phi \mapsto \phi^\vee $ was defined in Theorem~{\rm \ref{theorem-mainman}(1)}.
\end{Lemma}
\begin{figure}[h]
 \centering
 \includegraphics[scale=0.4]{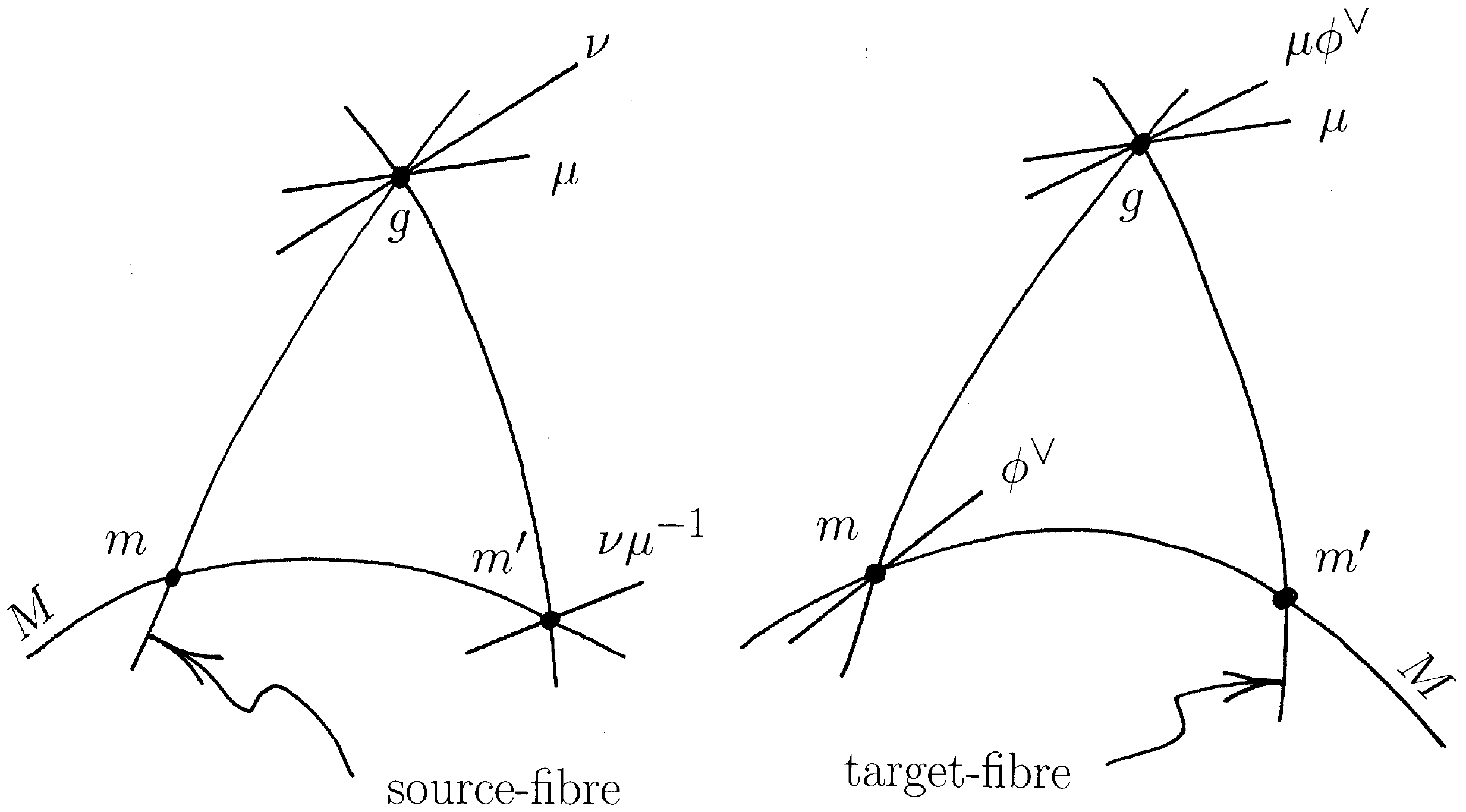}
 \caption{Schematic showing the f\/irst two products considered in Lemma~\ref{lemma-hotel}. In the lemma we view elements $\mu,\nu,\varcheck \phi \in J^1G$ as linear maps but in the f\/igure they are represented by their images (straight-line segments).}\label{fig1}
\end{figure}

\begin{proof}
We f\/irst prove (1). Let ${b}_1$ and ${b}_2$ be local bisections of $G$ (viewed as local right inverses of the source-projection $G \rightarrow M$) such that
 \begin{gather} \label{eq1}
 \mu = T_m {b}_1, \qquad \varcheck \phi= T_m {b}_2.
 \end{gather}
Let $v \in T_{m}M $ be arbitrary and $\gamma' \colon (-\epsilon,\epsilon)\rightarrow M$ be a simple regular path with $\gamma'(0)=m'$ and $\dot\gamma'(0)=v$, where a dot denotes derivative. According to the def\/inition of multiplication in $J^1G$, we have
 \begin{gather} \label{eq2}
 \mu \varcheck\phi (v)=\frac{d}{dt} g(t)h(t)\Big|_{t=0},
 \end{gather}
where
 \begin{gather} \label{eq3}
 g(t) := {b}_1(\gamma(t)),\qquad
 h(t) := {b}_2(\gamma'(t)),\qquad 
 \gamma(t) := \beta(h(t))=\beta({b}_2(\gamma'(t))).
 \end{gather}
 Here $\beta \colon G \rightarrow M$ denotes the target projection. We note in passing that
 \begin{gather}
 \dot \gamma (0)=\phi^{TM} v.\label{add}
 \end{gather}

Unfortunately, we cannot apply the chain rule directly to~\eqref{eq2} because $g(s)h(t)$ need not be def\/ined unless $s=t$. To overcome this dif\/f\/iculty, equip $G$ with an source connection $D$, as described in Section~\ref{source} above, and let
\begin{gather*}
(t_1,t_2,g) \mapsto A^\gamma_{t_1,t_2}(g) \colon \ (-\epsilon, \epsilon)^2 \times_\gamma U \rightarrow U
 \end{gather*}
denote the associated parallel action along $\gamma$. Here $\epsilon>0$, and we may arrange $U \subset G$ to be an open set containing both elements $m,g \in G$. We may now write
 \begin{gather*}
 g(t)h(t) = A^{\gamma}_{0,t} (\tilde g(t))\big(A^{\gamma}_{0,t}\big(\tilde h(t)^{-1}\big)\big)^{-1}
 \end{gather*}
 where
 \begin{gather}
 \label{a}
 \tilde g(t) := A^{\gamma}_{t,0}(g(t)),\\
 \label{b} 
 \tilde h(t) := \big(A^{\gamma}_{t,0}\big(h(t)^{-1}\big)\big)^{-1}.
 \end{gather}
 Noting that $g(t)h(t)=f(t,t)$, where
 \begin{gather*}
 f(s,t)= A^{\gamma}_{0,s} (\tilde g(t))\big(A^{\gamma}_{0,s}\big(\tilde h(t)^{-1}\big)\big)^{-1}
 \end{gather*}
 is well-def\/ined for all $(s,t) \in {\mathbb R}^2$ suf\/f\/iciently close to the origin, we may now apply the chain rule to \eqref{eq2}:
 \begin{align}
 \varcheck\phi \mu (v) &= \frac{{d} }{{d} s} f(s,0)\Big|_{s=0} + \frac{{d} }{{d} t} f(0,t)\Big|_{t=0}\nonumber\\
 &=\frac{{d} }{{d} s} A^{\gamma}_{0,s} (\tilde g(0))\big( A^{\gamma}_{0,s}\big(\tilde
 h(0)^{-1}\big) \big)^{-1} \Big|_{s=0} +\frac{{d} }{{d} t} A^{\gamma}_{0,0} (\tilde g(t))\big( A^{\gamma}_{0,0}\big(\tilde
 h(t)^{-1}\big) \big)^{-1} \Big|_{t=0}\nonumber\\
 &= w + \frac{{d} }{{d} t} \tilde g(t) \tilde h(t)\Big|_{t=0},\label{two}
 \end{align}
 where $w$ is the $D$-horizontal lift of $\dot \gamma(0)=\phi^{TM}v$ to an element of $D(g)\subset T_gG$.

Noting that $\tilde g(s) \tilde h(t)$ is def\/ined for all $(s,t) \in {\mathbb R}^2$ suf\/f\/iciently close to the origin, we may apply the chain rule directly to obtain
 \begin{gather} \label{three}
 \frac{{d} }{{d} t} \tilde g(t) \tilde h(t) \Big|_{t=0} = \dot{\tilde g}(0) + T L_g \dot{\tilde h}(0).
 \end{gather}
 It remains to compute $\dot{\tilde g}(0)$ and $\dot{\tilde h}(0)$.
 From~\eqref{a} we have
 \begin{gather*}
 g(t)= A^{\gamma}_{0,t}(\tilde g(t))=A^{\gamma}_{0,s}(\tilde g(t))\Big|_{s=t},
 \end{gather*}
 so that the chain rule gives
 \begin{gather*}
 \dot g(0)=\frac{{d} }{{d} s} A^{\gamma}_{0,s}(\tilde g(0))\Big|_{s=0} + \frac{{d} }{{d} t} A^{\gamma}_{0,0}(\tilde g(t))\Big|_{t=0}= w + \dot{\tilde g}(0).
 \end{gather*}
 Therefore
 \begin{gather*}
 \dot{\tilde g}(0) = \dot g(0) - w.
 \end{gather*}
 Similar arguments applied to \eqref{b} give us
 \begin{gather*}
 \dot{\tilde h}(0) = \dot h(0) - \phi^{TM}v.
 \end{gather*}
 These last two results and \eqref{three} imply
 \begin{gather*}
 \frac{{d} }{{d} t} \tilde g(t) \tilde h(t) \Big|_{t=0}= \dot g(0) - w + T L_g \cdot \big(\dot h(0) - \phi^{TM}v\big).
 \end{gather*}
 Equation \eqref{two} now reads
 \begin{gather*}
 \varcheck\phi \mu (v)=\dot g(0)+ T L_g \cdot \big(\dot h(0) - \phi^{TM} v\big).
 \end{gather*}
Finally, appealing to \eqref{eq3}, \eqref{add} and \eqref{eq1}, we obtain
 \begin{gather*}
 \varcheck\phi \mu (v) = \mu \big(\phi^{TM} v\big) + T L_g \cdot \big(\varcheck\phi(v) - \phi^{TM} v\big).
 \end{gather*}
 Noting the identity $\#X - X = T I \cdot X$, $X \in {\mathfrak g}$, one obtains{\samepage
 \begin{gather*}
 T L_g \cdot \big(\varcheck \phi (v) - \phi^{TM} v\big) =T L_g \cdot (\#\phi v - \phi v)=T L_g \cdot T I \cdot (\phi v).
 \end{gather*}
 Formula (1) now follows.}

The strategy for proving (2) is very similar and is omitted. Regarding (3), we have, putting $\nu = \mu \varcheck \phi $ in (2), $\mu \varcheck \phi \mu^{-1} =\varcheck \psi$, where
 \begin{gather*}
 \psi v = T R_{g^{-1}} \cdot \big(\mu(w)-(\mu\varcheck \phi)(w)\big),
 \end{gather*}
 and $w = \Adjoint_\mu^{-1} v$. Substituting~(1):
 \begin{align*}
 \psi v&=T R_{g^{-1}} \cdot \big(\mu(w) - \mu(\phi^{TM} w) - T L_g \cdot T I \cdot (\phi w)\big)\\
 &=T R_{g^{-1}} \cdot (\mu(\# \phi w) - T L_g \cdot T I \cdot (\phi w))\\
 &= \Adjoint_\mu(\phi w)=\Adjoint_\mu\big(\phi \Adjoint_\mu^{-1}v\big) = (\Adjoint_\mu \phi) v.
 \end{align*}

 Finally, we note that (4) is a straightforward consequence of (3) and (2).
\end{proof}

\subsection{Proof of Theorem \ref{theorem-mainman}}
That the map $\phi \mapsto \phi^\vee$ in Theorem~\ref{theorem-mainman}(1) is an embedding, with image the kernel of $J^1 G \rightarrow G$, is readily checked, as are the equivariance claims. To prove $\psi^\vee \phi^\vee = (\psi \phi)^\vee$, take $g=m$ and $\mu = \psi^\vee$ in Lemma~\ref{lemma-hotel}(1) and apply the identity $TI \cdot X = \#X -X$, for $X \in {\mathfrak g} $. Note that multiplication in $\model$ was def\/ined in~\eqref{clost}.

To prove Theorem~\ref{theorem-mainman}(2), let $b_1$, $b_2$ be bisections of $G$ and $\Phi_1$, $\Phi_2$ sections of $\model$. Let $m \in M$ be given and def\/ine $m'=\beta(b_2(m))$ and $m''= \beta(b_1(m'))$. We compute
\begin{align*}
 a \big( (b_1,\Phi_1)(b_2,\Phi_2) \big)(m) & =a \big( b_1b_2, \Phi_1~ b_1 \cdot \Phi_2 \big)(m)=\big((\Phi_1~b_1 \cdot \Phi_2)(m'') \big)^\vee T_m(b_1 b_2)\\
 &=\big( \Phi_1(m'') \Adjoint_{T_{m'}b_1}\Phi_2(m') \big)^\vee T_m(b_1b_2)\\
 &\overset{\text{by Theorem~\ref{theorem-mainman}(1)}}{=}\big( \Phi_1(m'') \big)^\vee \big( \Adjoint_{T_{m'}b_1}\Phi_2(m') \big)^\vee T_m(b_1b_2)\\
 &\overset{\text{by Lemma~\ref{lemma-hotel}(3)}}{=}\big( \Phi_1(m'') \big)^\vee T_{m'}b_1 \big( \Phi_2(m') \big)^\vee (T_{m'}b_1)^{-1} T_m(b_1b_2)\\
 &= \big( \Phi_1(m'') \big)^\vee T_{m'}b_1 \big( \Phi_2(m') \big)^\vee T_mb_2 = a (b_1,\Phi_1) a(b_2,\Phi_2).
\end{align*}

Regarding, Theorem~\ref{theorem-mainman}(3), the reader will be readily convinced that the map
\begin{gather*} c \colon \ G \times_\beta \model \rightarrow J^1 G\end{gather*}
def\/ined by $c(g,\phi)= \phi^\vee S(g)$ is a dif\/feomorphism. To show that it is a morphism of groupoids, let $(g_1,\phi_1), (g_2,\phi_2) \in G \times_\beta \model$ be given. Then
\begin{align*}
 c(g_1,\phi_1) c(g_2,\phi_2)&=\phi_1^\vee S(g_1)\phi_2^\vee S(g_2)=\phi_1^\vee S(g_1)\phi_2^\vee S(g_1)^{-1}S(g_1 g_2)\\
 &\overset{\text{by Lemma~\ref{lemma-hotel}(3)}}{=}\phi_1^\vee (\Adjoint_{S(g_1)}\phi_2)^\vee S(g_1 g_2) \\
 &\overset{\text{by Theorem \ref{theorem-mainman}(1)}}{=}(\phi_1\Adjoint_{S(g_1)}\phi_2)^\vee S(g_1 g_2)\\
 &= c(g_1 g_2, \phi_1\Adjoint_{S(g_1)}\phi_2)=c \big( (g_1,\phi_1)(g_2,\phi_2) \big).
\end{align*}

\section{On the inf\/initesimalization of Cartan connections}\label{gg}
Let $D \subset TG$ be a Cartan connection on a Lie groupoid $G$ over $M$. Then, in the terminology of Section~\ref{source}, $D$ is also a~globally def\/ined source-connection on $G$. We dif\/ferentiate the corresponding parallel action of $D$, along curves in~$M$, to obtain a~linear connection $\nabla^D$ on the Lie algebroid~$\mathfrak g$ of $G$. Theorem~\ref{theorem-aswedid} below states that this `geometrically' def\/ined connection $\nabla^D$ coincides with the inf\/initesimal Cartan connection~$\nabla $, def\/ined `algebraically' in Section~\ref{hsd}.

\subsection{Dif\/ferentiating the parallel action def\/ined by a source connection}\label{aswedid}
Let $D$ be an arbitrary source connection on $G$. Let $v \in T_mM$ be an arbitrary tangent vector at some point $m \in M$, and $\gamma \colon (-\epsilon,\epsilon)\rightarrow M$ a simple regular path with $\gamma(0)=m$ and $\dot\gamma(0)=v$, where a dot denotes derivative. Shrinking $\epsilon $ if necessary, we have, for some open neighborhood~$U$ of~$m$ in~$\gamma^*G$, the corresponding parallel action $A^{\gamma}\colon (-\epsilon ,\epsilon)^2\times_\gamma U \rightarrow U$ def\/ined in Section~\ref{source}~-- an action of the pair groupoid $(-\epsilon,\epsilon)^2$ on $U$. We obtain an action $a^\gamma$ of $(-\epsilon,\epsilon)^2$ on the pullback~$\gamma^* {\mathfrak g}$~-- which, for simplicity, we identify with a submanifold of
${\mathfrak g}$~-- by dif\/ferentiating:
\begin{gather*}
 a^\gamma_{t_1,t_2}\left( \frac{d}{ds}g(s)\Big|_{s=0} \right):= \frac{d}{ds}A^\gamma_{t_1,t_2}g(s)\Big|_{s=0}.
\end{gather*}
Here we view an element of $\gamma^* {\mathfrak g}|_m\cong {\mathfrak g}|_m$ as the derivative of some path $s \mapsto g(s) \in G$ lying in a~source-f\/ibre of $G$. A linear connection $\nabla^D$ on ${\mathfrak g} $ is now def\/ined by
\begin{gather*}
 \nabla^D_vX := \partial_t a^\gamma_{t,0}X(\gamma(t))|_{t=0}, \qquad X \in \Gamma({\mathfrak g}).
\end{gather*}
\begin{Notation}
Here and in the sequel $\partial_t$ indicates a derivative with respect to~$t$ of a vector-valued function of $t$, identif\/ied with an element of the underlying vector space and not a tangent vector (i.e., with base-point `forgotten').
\end{Notation}

Here is the main result of the present section:
\begin{Theorem}\label{theorem-aswedid}
If $D \subset TG$ is a Cartan connection on $G$ then $\nabla^D$ coincides with the correspon\-ding infinitesimal Cartan connection~$\nabla$ defined in Section~{\rm \ref{hsd}}.
\end{Theorem}

The proof of this theorem is given in Section~\ref{later} after necessary preparations.

\subsection[The Lie algebroid of $\model$]{The Lie algebroid of $\boldsymbol{\model}$}\label{onthe}
Implicit in the discussion at the beginning of Section~\ref{sec3} is an identif\/ication of the Lie algebroid of our model $\model$ of the kernel of $J^1 G \rightarrow G$ with the kernel $T^*\!M \otimes {\mathfrak g} $ of $J^1 {\mathfrak g} \rightarrow {\mathfrak g}$. To prove the preceding theorem we must make this identif\/ication explicit. To this end, let ${\mathcal L}(G)$ denote the (abstract) Lie algebroid of any Lie groupoid $G$. Then an element of ${\mathcal L}(\model)$ is of the form $\frac{d}{dt}\phi_t|_{t=0}$, for some path $t \mapsto \phi_t \in \model$ lying completely in some source-f\/ibre (=~target-f\/ibre) of $\model$; the latter is simply an open neighborhood of zero in a vector space
\begin{gather*} (T^*\!M \otimes {\mathfrak g})|_m,\qquad m \in M.\end{gather*}
Of course, we also require that $\phi_0$ be an identity element of $\model$, which just means that $\phi_0$ is the zero element of $(T^*\!M \otimes {\mathfrak g})|_m$. There is a natural isomorphism of vector bundles $\theta' \colon {\mathcal L}(\model) \rightarrow T^*\!M \otimes {\mathfrak g}$ def\/ined by
\begin{gather*}
 \theta'\left( \frac{d}{dt}\phi_t\Big|_{t=0} \right) v =\partial_t \phi_t v|_{t=0}. 
\end{gather*}
This isomorphism is the appropriate one for present purposes, by virtue of the following:
\begin{Proposition}\label{proposition-onthe}
 Defining $\theta'$ as above, we have the following commutative diagram with exact rows:
 \begin{gather*}
 \begin{CD}
 0 @>>> {\mathcal L}(\model) @>j'>> {\mathcal L}(J^1 G) @>>> {\mathcal L}(G) @>>> 0\\
 @. @V\theta' VV @V\theta VV @| @.\\
 0 @>>> T^*\!M \otimes {\mathfrak g} @>j>> J^1 {\mathfrak g}
 @>>>{\mathfrak g} @>>> 0.
 \end{CD}%
 \end{gather*}
Here $\theta$ denotes the isomorphism described in Section~{\rm \ref{repsagain}}; $j'$ denotes the derivative of the map $\phi \mapsto \phi^\vee \colon \model \rightarrow J^1 G$ defined in Theorem~{\rm \ref{theorem-mainman}(1)}; $j$ denotes the map which induces a~corresponding map of section spaces given by $df \otimes X \mapsto f J^1 X - J^1 (fX)$.
\end{Proposition}

Key in the proof of the proposition is the following technical result applied again in Section~\ref{later} below:
\begin{Lemma}\label{lemma-onthe}
Let $\alpha \colon G \rightarrow M$ denote the source map of $G$ and let $D $ be any source-connection. For each $g \in G$ beginning at some $m=\alpha(g)$, let $S(g) \colon T_{m}M \rightarrow D(g)$ denote the inverse of the restriction $T_g \alpha \colon D(g) \rightarrow T_mM$ of the tangent map $T \alpha $. $($In other words, $S(g)(v)$ is the $D$-horizontal lift of $v \in T_{m}M$ to a vector at $g$; if $D$ is a Cartan connection, we view $S(g) \in J^1 G$ and $S\colon G \rightarrow J^1 G$ has the meaning in Section~{\rm \ref{onepoint}.)} Then, for any function $f$ on $G$ defined in a~neighborhood of~$m$, one has
\begin{gather*}
 \big\langle df, \nabla^D_vX\big\rangle= \partial_t \big\langle df, T_m\big( \Phi_{X^\mathrm{R}}^t \circ
 \iota_M \big)\cdot v \big\rangle\big|_{t=0} - \partial_t \big\langle df, S\big( \Phi_{X^\mathrm{R}}^t(m) \big)(v) \big\rangle\big|_{t=0}.
 \end{gather*}
\end{Lemma}
\begin{proof}
Adopting the notation of Section~\ref{aswedid} with $v=\dot \gamma (0)$, we calculate,
 \begin{align*}
 \big\langle df, \nabla^D_vX\big\rangle
 &=\partial_s\big\langle df, a^\gamma_{s,0}X(\gamma(s)) \big\rangle\big|_{s=0}
 =\partial_s \partial_t f\big( A^\gamma_{s,0}\Phi_{X^\mathrm{R}}^t(\gamma(s)) \big)\big|_{t=0,s=0} \\
 &=\partial_t \partial_s f(g(s,t))|_{t=0,s=0},
 \end{align*}
i.e.,
\begin{gather}
\big\langle df, \nabla^D_vX\big\rangle =\partial_t \left\langle df, \frac{\partial }{\partial
 s}g(s,t)\Big|_{s=0} \right\rangle\Big|_{t=0},\label{house}
 \end{gather}
where $g(s,t)=A^\gamma_{s,0}\Phi_{X^\mathrm{R}}^t(\gamma(s))$ is def\/ined for all $(s,t)$ suf\/f\/iciently close to $(0,0)$, and we have appealed to the equality of mixed partial derivatives. Now we can write $\Phi_{X^\mathrm{R}}^t (\gamma(s))=A^\gamma_{0,s}g(s,t)$ and notice that $A^\gamma_{0,s_1}g(s_2,t)$ is def\/ined for all $(s_1,s_2,t)$ suf\/f\/iciently close to $(0,0,0)$. It follows that we may dif\/ferentiate with respect to $s$ at $s=0$ and apply the chain rule:
 \begin{gather*}
 \frac{\partial }{\partial s}\Phi_{X^\mathrm{R}}^t (\gamma(s))\Big|_{s=0}=\frac{\partial }{\partial s}A^\gamma_{0,s}g(0,t)\Big|_{s=0}+
 \frac{\partial }{\partial s}A^\gamma_{0,0}g(s,t)\Big|_{s=0},
 \end{gather*}
i.e.,
\begin{gather*} T_m\big(\Phi_{X^\mathrm{R}}^t \circ \iota_M\big)\cdot v= S(\Phi_{X^\mathrm{R}}^t)(v) + \frac{\partial }{\partial s}g(s,t)\Big|_{s=0}.
 \end{gather*}
 We conclude that
 \begin{gather*}
 \frac{\partial }{\partial s}g(s,t)\Big|_{s=0}=T_m\big(\Phi_{X^\mathrm{R}}^t \circ \iota_M\big) \cdot v - S\big(\Phi_{X^\mathrm{R}}^t(m)\big)(v).
 \end{gather*}
 Substituting this into \eqref{house} gives the desired result.
\end{proof}

\begin{proof}[Proof of Proposition~\ref{proposition-onthe}] Exactness of the top row follows from Theorem~\ref{theorem-mainman}(1); exactness of the bottom row is immediate. It remains to show that for any $\zeta \in {\mathcal L} (\model)$, we have
 \begin{gather}
 \theta(j'(\zeta ))=j(\theta'(\zeta)).\label{coddleston}
 \end{gather}

 By exactness, and the fact that $\theta $ is an isomorphism, the element $j'(\zeta)$ is in the image of the composite morphism
 \begin{gather*}
 T^*\!M \otimes {\mathfrak g} \xrightarrow{j} J^1 {\mathfrak g} \xrightarrow{\theta^{-1}} {\mathcal L} \big(J^1 G\big).
 \end{gather*}
 But any element of $T^*\!M \otimes {\mathfrak g} $ is of the form $-\nabla X(m)$, where $\nabla $ is an {\em arbitrary} linear connection on~$\mathfrak g$, $m \in M$, and $X$ is a local section of ${\mathfrak g} $ with $X(m)=0$. Since in that case $j(-\nabla X (m))=J^1_mX$, it follows, from formula~\eqref{sop} for $\theta^{-1}$ that
 \begin{gather}
 j'(\zeta)=\frac{d}{dt}T_m\big(\Phi_{X^\mathrm{R}}^t\circ \iota_M\big)\Big|_{t=0},\label{occ}
 \end{gather}
for some $m \in M$ and $X \in \Gamma({\mathfrak g})$ with $X(m)=0$. Furthermore, we have, by construction, $\theta(j'(\zeta))=j(-\nabla X(m))$. Consequently, to prove~\eqref{coddleston}, it suf\/f\/ices, since $j$ is injective, to show that
 \begin{gather}
 \nabla X(m)=-\theta'(\zeta).\label{goop}
 \end{gather}

Now the map $j'$ is the derivative of the embedding $\phi \mapsto \phi^\vee \colon \model \rightarrow J^1 G$ def\/ined in Theorem~\ref{theorem-mainman}(1). In view of this fact and~\eqref{occ}, there is a path $t \mapsto \phi_t$ in $\model$ such that $\phi_t^\vee = T_m(\Phi_{X^\mathrm{R}}^t\circ \iota_M)$ and $\zeta=\frac{d}{dt}\phi_t|_{t=0}$. To prove~\eqref{goop} we will apply both sides to some tangent vector $v = \dot \gamma(0) \in T_mM$ and test the resulting elements of ${\mathfrak g}|_m$ on a test function~$f$, def\/ined on~$G$, in some neighborhood of $m$. It will be convenient to take $\nabla=\nabla^D$, where $D \subset TG$ is an arbitrary source connection on $G$. Then, applying the preceding lemma, and using the fact that $\Phi^t_{X^{\mathrm R}}(m)=m$ (since $X(m)=0$) we compute
\begin{align*}
 \langle df, \nabla_vX\rangle=\langle df, \nabla^D_vX\rangle &=\partial_t\big\langle df, T_m(\Phi^t_{X^{\mathrm R}}\circ
 \iota_M)\cdot v \big\rangle\big|_{t=0} - \partial_t\langle df, v\rangle |_{t=0}\\
 &=\partial_t\big\langle df, \phi_t^\vee(v) - v \big\rangle\big|_{t=0} =\partial_t\big\langle df, -\phi_t v \big\rangle\big|_{t=0}\\
 &=-\langle df, \partial_t(\phi_t v) \rangle|_{t=0} =-\left\langle df, \theta' \left( \frac{d }{dt}\phi_t\Big|_{t=0} \right)v \right\rangle\\
 & =-\langle df,\theta'(\zeta)v\rangle,
 \end{align*}
 which proves \eqref{goop}.
\end{proof}

\subsection{The proof of Theorem~\ref{theorem-aswedid}}\label{later}
\begin{Lemma}\label{lemma-later}
Let $g_t$ be a path in $G$ such that $g_0=m \in M$ and $g_t$ begins at $m$ for all $t$. Suppose $\mu_t$ and $\nu_t$ are two paths in $J^1G$ with common projection $g_t$ onto $G$, and such that $\mu_0$ and $\nu_0$ coincide with the identity element $T_m\iota_M$ of $J^1 G$ $($see Fig.~{\rm \ref{fig2})}. Then there exists a (necessarily unique) $\psi \in T^*\!M \otimes {\mathfrak g} $ such that:
 \begin{enumerate}\itemsep=0pt
 \item[$(1)$]
 $\theta \left( \frac{d}{dt}\mu_t\big|_{t=0} - \frac{d}{dt}\nu_t\big|_{t=0} \right)=j(\psi)$, and
 \item[$(2)$]
 $\langle df, \psi v \rangle=-\partial_t \langle df, \mu_t(v)\rangle|_{t=0} + \partial_t \langle df, \nu_t(v)\rangle|_{t=0}$,
 \end{enumerate}
 for all functions $f$ on $G$ defined in some neighborhood of $m$, and all $v \in T_mM$.
\end{Lemma}

For the def\/inition of $\theta$ and $j$ see Proposition~\ref{proposition-onthe}.
\begin{figure}[h] \centering
 \includegraphics[scale=0.35]{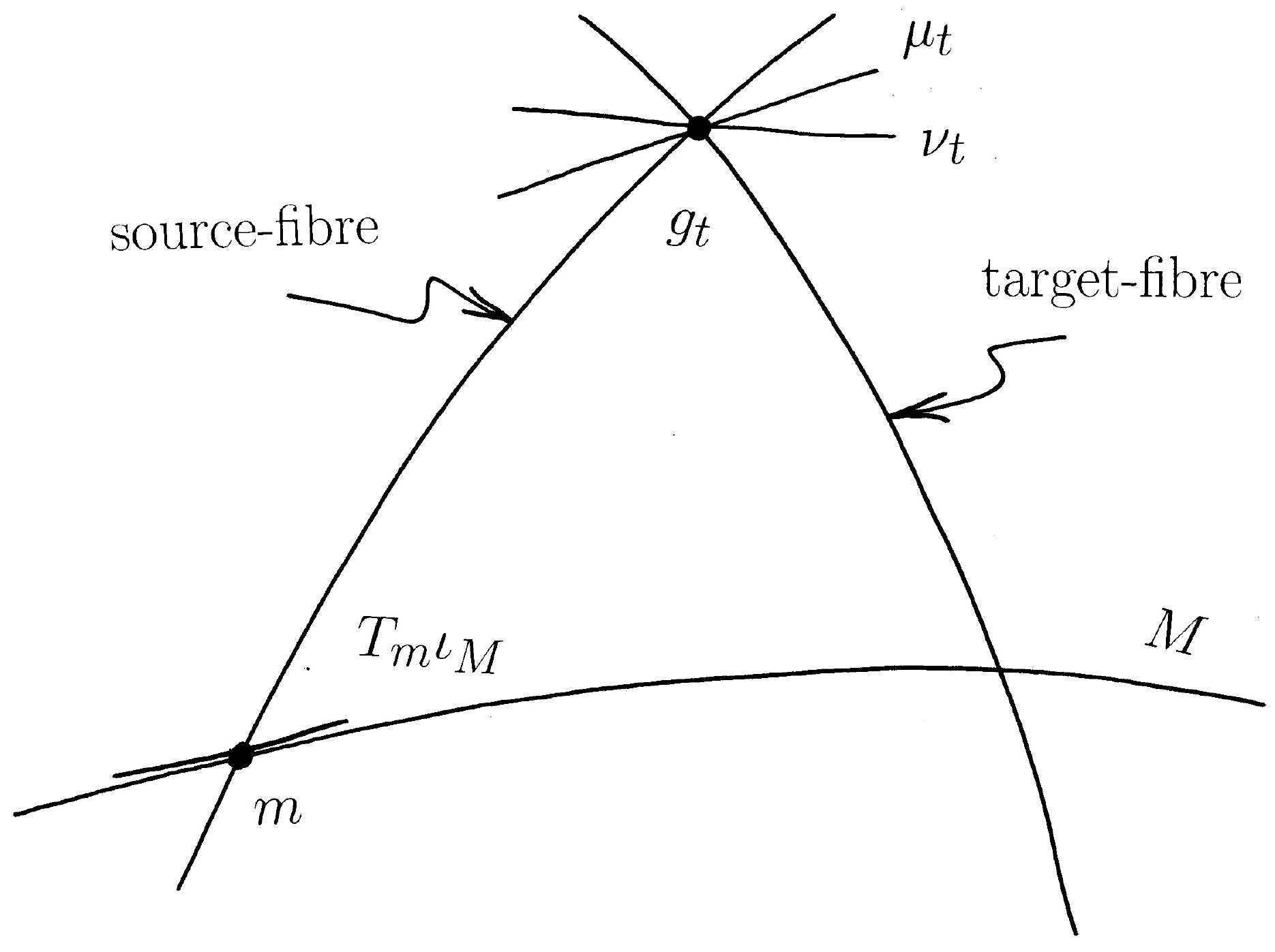}
 \caption{Schematic showing the paths in $J^1 G$ considered in Lemma~\ref{lemma-later}.}\label{fig2}
\end{figure}

\begin{proof}[Proof of Lemma~\ref{lemma-later}]
 Since $\mu_t$ and $\nu_t$ have common projection onto $G$, we have $\nu_t=\varcheck{\phi_t} \mu_t$ (multiplication in $J^1 G$) for some $\phi_t \in T^*\!M \otimes {\mathfrak g}$. We have $\phi_0^\vee=T_m \iota_M$, i.e., $\phi_0=0$. Dif\/ferentiating and applying the chain rule, we obtain
\begin{gather*}
 \frac{d}{dt}\mu_t\Big|_{t=0} - \frac{d}{dt}\nu_t\Big|_{t=0} = -\frac{d}{dt}\varcheck{\phi_t}\Big|_{t=0}= -j' \left( \frac{d}{dt}\phi_t\Big|_{t=0} \right),
\end{gather*}
where $j'$ is def\/ined in Proposition~\ref{proposition-onthe}. Applying $\theta $ to both sides, and using commutativity of the diagram appearing in that proposition, we obtain~(1) if we take $\psi=-\theta'(\frac{d}{dt}\phi_t)$. In that case, by the def\/inition of $\theta'$, we have
\begin{gather}
 \psi v = -\partial_t (\phi_t v)|_{t=0},\qquad v \in T_mM.\label{supermarket}
\end{gather}

On the other hand, applying Lemma~\ref{lemma-hotel}(4), we have
\begin{gather*}
 TR_{g_t }\cdot \Adjoint_{\mu_t} (\phi_t v)=\mu_t(v)-\nu_t(v).
\end{gather*}
Applying $df $ to both sides and dif\/ferentiating using the chain rule on the left, we obtain
\begin{gather*}
 \partial_t\langle df, TR_{g_t }\cdot \Adjoint_{\mu_t}\zeta_{TG}(m)\rangle|_{t=0} +~ \partial_t\langle df, \phi_t v\rangle|_{t=0}
 =\partial_t \langle df, \mu_t(v)\rangle|_{t=0} - \partial_t \langle df, \nu_t(v)\rangle|_{t=0}.
\end{gather*}
Here $\zeta_{TG}$ denotes the zero-section of $TG$. We have $TR_{g_t}\cdot \Adjoint_{\mu_t}\zeta_{TG}(m)=\zeta_{TG}(g_t)$, which means the f\/irst term on the left vanishes, giving us
\begin{gather*}
 \langle df, \partial_t(\phi_t v)|_{t=0}\rangle =\partial_t \langle df, \mu_t(v)\rangle|_{t=0} - \partial_t \langle df, \nu_t(v)\rangle|_{t=0}.
\end{gather*}
This, together with \eqref{supermarket}, establishes (2).
\end{proof}

Now suppose $D \subset TG$ is a Cartan connection on $G$, let $S\colon G \rightarrow J^1 G$ denote the corresponding morphism of Lie groupoids, and $\nabla$ the corresponding inf\/initesimal Cartan connection on~\mbox{${\mathfrak g} ={\mathcal L}(G)$}. We are now in a~position to prove $\nabla^D=\nabla$.

First, let us make the `algebraic' def\/inition of $\nabla$ in Section~\ref{hsd} more explicit. If we denote the (abstract) derivative of $S \colon G \rightarrow J^1 G$ by $dS \colon {\mathcal L}(G) \rightarrow {\mathcal L}(J^1 G)$ then we obtain a splitting $s \colon {\mathfrak g} \rightarrow J^1 {\mathfrak g}$ of the exact sequence appearing in the bottom row of the commutative diagram in Proposition~\ref{proposition-onthe}, by putting $s=\theta \circ dS$. Then, $\nabla $ is def\/ined implicitly by
\begin{gather*}
 j(\nabla X(m))=-J^1_mX +s(X(m)).
\end{gather*}
Recalling the def\/inition of $\theta $ given in~\eqref{sop}, we may rewrite this as
\begin{gather*}
 j(\nabla X(m))=\theta \left( -\frac{d}{dt}T_m\big( \Phi^t_{X^\mathrm{R}}\circ \iota_M \big)\Big|_{t=0} + \frac{d}{dt}S\big( \Phi^t_{X^\mathrm{R}}(m) \big)\Big|_{t=0} \right).
\end{gather*}
Applying the preceding lemma, we obtain
\begin{gather*}
 \langle df, \nabla_vX \rangle=\partial_t \big\langle df, T_m\big( \Phi^t_{X^\mathrm{R}}\circ \iota_M \big) \cdot v \big\rangle|_{t=0} -
 \partial_t \big\langle df, S\big( \Phi_{X^\mathrm{R}}^t(m) \big)(v) \big\rangle\big|_{t=0}.
\end{gather*}
Here $v \in T_mM$ is arbitrary, as is the locally def\/ined function~$f$ on~$G$. Our claim $\nabla=\nabla^D$ now follows from Lemma~\ref{lemma-onthe}.

\section{Proof of Theorem \ref{theorem1.2}}\label{srt}
In this section $D$ is a Cartan connection on $G$, $S \colon G \rightarrow J^1 G$ the corresponding right-inverse for the natural projection $J^1 G \rightarrow G$, and $\nabla $ the corresponding inf\/initesimal Cartan connection.

\subsection{Proof of Theorem \ref{theorem1.2}(1)}
Let ${\mathcal F} $ be the foliation on $V \subset G$ integrating $D$. For each $m \in M$ let $P_m \subset V$ be the intersection of $V$ with the source f\/ibre of $G$ over $m$. Then $P_m$ is transverse to ${\mathcal F} $. In some open neighborhood $U \subset G$ of $m_0$, the leaves of ${\mathcal F}$ intersect $U$ in level sets of a~submersion $\Omega \colon U \rightarrow P_{m_0}$. Now let $v=\dot \gamma (0) \in T_mM$ be the derivative of some path~$\gamma $ on~$M$, $m \in M\cap U$. Then, denoting by $A^\gamma $ the parallel action along $\gamma $ determined by $D$, we have
\begin{gather}
 \Omega \big( A^\gamma_{t,0}(g)\big) = \Omega(g), \label{bits}
\end{gather}
for all $t$ suf\/f\/iciently close to zero, and all $g \in P_{\gamma(t)}$ suf\/f\/iciently close to $\gamma(t)\in M$.

If ${\mathfrak g}|_{M\cap U}$ denotes the restriction of ${\mathfrak g}$, and ${\mathfrak g}_0:=T_{m_0}P_{m_0}$, then we obtain a~map $\omega \colon {\mathfrak g}_{M\cap U} \rightarrow {\mathfrak g}_0 $ by dif\/ferentiating $\Omega \colon U \rightarrow P_{m_0}$:
 \begin{gather}
 \omega \left( \frac{d}{ds} g(s)\Big|_{s=0}\right) :=\frac{d}{ds}\Omega \big( g(s) \big) \Big|_{s=0}.\label{tcal}
 \end{gather}
 As usual, we are viewing an element of ${\mathfrak g}$ as an equivalence class of paths in a source f\/ibre. Applying Theorem~\ref{theorem-aswedid} we calculate
\begin{align*}
 \omega(\nabla_vX)&=\omega\left( \partial_t \frac{\partial }{\partial s} A^\gamma_{t,0}\big( \Phi_{X^\mathrm{R}}^s(\gamma(t)) \big)\big|_{s=0,t=0} \right)\\
 & \overset{\text{by \eqref{tcal}}}{=}\partial_t \frac{\partial }{\partial s}\Omega\big(
 A^\gamma_{t,0}\big( \Phi_{X^\mathrm{R}}^s(\gamma(t)) \big) \big) \Big|_{s=0,t=0}\enspace \\
 &\overset{\text{by \eqref{bits}}}{=}\partial_t \frac{\partial }{\partial s}\Omega\big(
 \Phi_{X^\mathrm{R}}^s(\gamma(t)) \big) \Big|_{s=0,t=0}\enspace \\
 &= \partial_t \omega \big( X(\gamma(t)) \big)|_{t=0} = \langle d(\omega(X)),v \rangle.
\end{align*}
So $\omega(\nabla_V X)={\mathcal L}_V (\omega(X))$, for any vector f\/ield $V$ on $M \cap U$, where ${\mathcal L} $ denotes Lie derivative. For any two vector f\/ields $V_1$ and $V_2$ on $M \cap U$ we deduce
\begin{gather*}
 \omega\big( \curvature \nabla (V_1,V_2)X \big)={\mathcal L}_{V_1} {\mathcal L}_{V_2} (\omega(X)) - {\mathcal L}_{V_2} {\mathcal
 L}_{V_1} (\omega(X)) - {\mathcal L}_{[V_1,V_2]}(\omega(X)) = 0.
\end{gather*}
This shows that $\nabla $ is f\/lat in a neighborhood of the arbitrarily f\/ixed point $m_0 \in M$.

The remaining sections are devoted to the proof of part (2) of Theorem \ref{theorem1.2}.

\subsection[Proof that $G^D \subset G$ is a subgroupoid]{Proof that $\boldsymbol{G^D \subset G}$ is a subgroupoid}\label{jgd}
This follows immediately from the following lemma (used again in Section~\ref{jgda} below):
\begin{Lemma}\label{lemma-jgd}
 Let $b_1 \colon U_1 \rightarrow G$ and $b_2 \colon U_2 \rightarrow G$ be two multipliable local bisections integra\-ting~$D$. Then their product $b_1 b_2$ integrates $D$.
\end{Lemma}
\begin{proof}
Let $m$ be an arbitrary point in $U_2$ (the domain of the product $b_1 b_2$) and let $g_1=b_1(m)$ and $g_2=b_2(m)$. The arrow $g_2$ terminates at some point $m'$ in the domain of $b_1$. Since $b_1$ and $b_2$ integrate $D$, the one-jet $S(g_2)$ of a local bisection at $m$ has representative $b_2$, while the one-jet~$S(g_1)$ of a~local bisection at $m'$ has representative $b_1$. By the def\/inition of the product in~$J^1 G$, $S(g_1)S(g_2)$ is the one-jet at~$m$ of the product $b_1 b_2$, a one-jet that must coincide with~$S(g_1g_2)$, because $D$ is a Cartan connection. So the image of the tangent map of~$b_1 b_2$ at~$m$ must be~$D(g_1g_2)$. Since $m$ was arbitrary, this shows $b_1b_2$ integrates $D$.
\end{proof}

Since the identity bisection lies in $G^D$, $G^D$ is wide in $G$.

From now on, we suppose that $\nabla $ is f\/lat.

\subsection[Proof that $G^D$ contains an open neighborhood of $M$]{Proof that $\boldsymbol{G^D}$ contains an open neighborhood of $\boldsymbol{M}$}\label{dpo}
Our proof that $G^D \subset G$ is open will rest on the fact that each point $m \in M$ has an open neighborhood lying in~$G^D$. In fact, we now prove the following stronger result:
\begin{Proposition}\label{proposition-dpo}
Assume $M$ is simply-connected and that $G$ has simply-connected source-fibres. Then, assuming~$\nabla$ is flat, $G^D=G$.
\end{Proposition}

We begin by recalling that every Lie algebroid equipped with a f\/lat Cartan connection $\nabla$ is an action algebroid, if $M$ is simply-connected \cite{Blaom_06} (for generalizations, see~\cite{Blaom_13}):
\begin{Lemma}\label{lemma-dpo}
Let ${\mathfrak g} $ be a Lie algebroid over a simply-connected manifold $M$ supporting a flat Cartan connection $\nabla$. Then the subspace ${\mathfrak g}_0 \subset \Gamma({\mathfrak g})$ of $\nabla $-parallel sections is a Lie subalgebra acting on $M$, with the corresponding Lie algebra homomorphism $\xi \mapsto \xi^\dagger \colon {\mathfrak g}_0 \rightarrow \Gamma({TM})$ being given by $\xi^\dagger(m):=\#\xi(m)$. Moreover, the vector bundle morphism $(\xi,m) \rightarrow \xi(m) \colon {\mathfrak g}_0 \times M \rightarrow {\mathfrak g}$ is a Lie algebroid isomorphism between the action algebroid ${\mathfrak g}_0 \times M$ and ${\mathfrak g} $.
\end{Lemma}

\begin{proof}[Proof of Proposition~\ref{proposition-dpo}]
Applying the lemma to the Lie algebra ${\mathfrak g} $ of $G$, we obtain a Lie algebra ${\mathfrak g}_0$ and a morphism of Lie algebroids $\omega \colon {\mathfrak g} \rightarrow {\mathfrak g}_0$, such that $\omega $ restricted to any f\/ibre~${\mathfrak g}|_m$ is an isomorphism onto ${\mathfrak g}_0$. Since $G$ is source-simply-connected, there exists a Lie groupoid morphism $\Omega \colon G \rightarrow G_0$ integrating $\omega \colon {\mathfrak g} \rightarrow {\mathfrak g}_0$, by Lie~II for Lie groupoids (see, e.g., \cite{Crainic_Fernandes_11}). Here~$G_0$ is the simply-connected Lie group integrating ${\mathfrak g}_0$. Because $\omega$ is a point-wise isomorphism, the restriction of $\Omega$ to any source-f\/ibre of $G$, or any target-f\/ibre of $G$, is a local dif\/feomorphism. In particular, each level set $\Omega^{-1}(g_0)$ of $\Omega$ is a pseudotransformation in the sense of Section~\ref{ergs}. Moreover, as $\Omega$ is a morphism of Lie groupoids, it is not hard to see that the foliation of~$G$ by connected components of these level sets is a pseudoaction of~$G$. The tangent $n$-plane f\/ield~$D'$ is accordingly a Cartan connection on~$G$, by Proposition~\ref{proposition-ergs}. Since~$D'$ is integrable by construction, to prove the proposition, it suf\/f\/ices to show $D=D'$.

Now $D=D'$ precisely when the corresponding Lie groupoid morphisms
 \begin{gather*}S,S' \colon \ G \rightarrow J^1 G\end{gather*}
are the same. Since $G$ has connected source-f\/ibres, it actually suf\/f\/ices to show that the derivatives $s, s' \colon {\mathfrak g} \rightarrow J^1 {\mathfrak g} $ coincide, or equivalently, that the corresponding inf\/initesimal Cartan connections~$\nabla$,~$\nabla'$ are the same.

Now $\nabla$ is f\/lat by hypothesis, and $\nabla'$ is f\/lat by Theorem~\ref{theorem1.2}(1), already proven above. Since~$M$ is simply-connected, to show $\nabla'=\nabla$ it is therefore suf\/f\/icient to show that $\nabla $-parallel sections of~${\mathfrak g} $ are also $\nabla' $-parallel. Now $X \in \Gamma({\mathfrak g})$ is $\nabla$-parallel precisely when~$\omega(X)$ is a constant~$\xi$. Let $v=\dot \gamma(0)\in T_mM$ be the derivative of some path~$\gamma$ on~$M$. Also, let $A^\gamma$ denote the parallel action along $\gamma $ def\/ined by the connection $D'$. Then, by the def\/inition of $D'$, we have
\begin{gather}
 \Omega \big(A^\gamma_{t,0}(g)\big) = \Omega(g), \label{bits2}
\end{gather}
for all $t$ suf\/f\/iciently close to $0$ and all $g \in \alpha^{-1}(\gamma(t))$ suf\/f\/iciently close to $\gamma(t) \in M$. On the other hand, applying Theorem~\ref{theorem-aswedid} to the connection $\nabla'$, we have
\begin{align*}
 \omega(\nabla'_vX)&=\omega\left( \partial_t \frac{\partial }{\partial s} A^\gamma_{t,0}\big( \Phi_{X^\mathrm{R}}^s(\gamma(t)) \big)\big|_{s=0,\,t=0} \right)
 =\partial_t \frac{\partial }{\partial s}\Omega\big( A^\gamma_{t,0}\big( \Phi_{X^\mathrm{R}}^s(\gamma(t)) \big) \big)
 \Big|_{s=0,t=0}\\
 &\overset{\text{by \eqref{bits2}}}{=}\partial_t \frac{\partial }{\partial s}\Omega\big( \Phi_{X^\mathrm{R}}^s(\gamma(t)) \big) \Big|_{s=0,\,t=0}\enspace
= \partial_t \omega \big( X(\gamma(t)) \big)|_{t=0} = \partial_t \xi |_{t=0} = 0.\tag*{\qed}
\end{align*}\renewcommand{\qed}{}
\end{proof}

\subsection[Proof that $G^D \subset G$ is open]{Proof that $\boldsymbol{G^D \subset G}$ is open}\label{jgda}
To show $G^D \subset G$ is open, let $g_0 \in G^D $ be any arrow, beginning at some $m_0 \in M$. By Section~\ref{dpo}, $m_0$ lies in some open neighborhood of $G$ contained entirely in $G^D$. In this neighborhood the connection $D$ is tangent to some foliation~${\mathcal F} $ of the neighborhood.

Using the fact that the foliation ${\mathcal F}$ is locally f\/ibrating, and using the fact that the restriction of the target projection $\beta \colon G \rightarrow M$ to each leaf of ${\mathcal F}$ is a~local dif\/feomorphism, it follows from the implicit function theorem that we can f\/ind a chart $\psi \colon U \times V \rightarrow G$, for a~neighborhood of~$m_0$ in~$G$, simultaneously adapted to the foliation and the target projection~$\beta $. Indeed, we may take~$U$ to be an open subset of $m_0$ in $M$, $V$ to be an open subset of the target-f\/ibre $\beta^{-1}(m_0)$, and may arrange that: (i)~$\beta(\psi(m,p))=m$; and (ii)~$\psi(U \times \{p\})$ is an integral manifold of $D$ through~$p$, and the image of a~local bisection of~$G$, for every $p \in V$. See Fig.~\ref{fig3}.
\begin{figure}[h]
\centering
\includegraphics[scale=0.18]{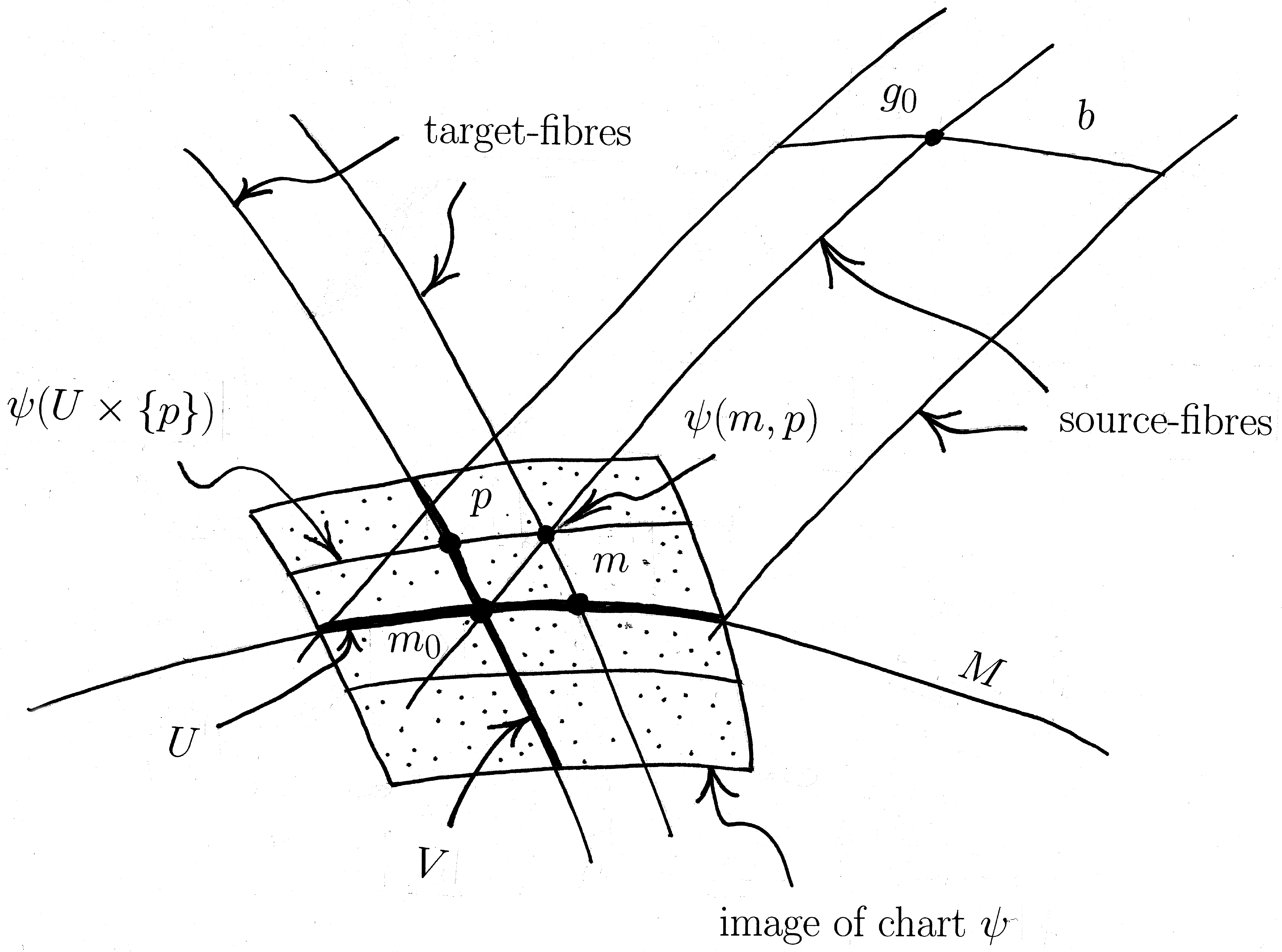}
 \caption{A chart $\psi \colon U \times V \rightarrow G$ for a~neighborhood of $m_0 \in M$.}\label{fig3}
\end{figure}

Since $g_0 \in G^D$, there is a local bisection $b$ of $G$ integrating $D$ with $b(m_0)=g_0$ (see Fig.~\ref{fig3}). Shrinking neighborhoods if necessary, we may suppose $b$ has domain $U$. We now use multiplication in the groupoid to translate the chart $\psi $ to a~chart $\bar\psi \colon U \times V\rightarrow G$ whose image is a~neighborhood of $g_0$: Put $\bar \psi(m,p)= b(m)\psi(m,p)$, which is well-def\/ined on account of (i) above. Now $\bar \psi $ maps $\{m\} \times V$ dif\/feomorphically into a f\/ibre of the target projection. On the other hand, for each f\/ixed $p \in V$, the map $m \mapsto \bar \psi(m,p) \colon U \rightarrow G$ is an immersion whose image $B_p$ is the product of the local bisection $b(U)$ and the local bisection $\psi (U \times \{p\})$ (here viewing local bisections as submanifolds of $G$). The two preceding remarks show that $\bar \psi $ has a tangent map of full rank, so that $\bar \psi(U \times V)$ is an open neighborhood of $g_0$. Furthermore, as the local bisection~$B_p$ is a product of local bisections integrating $D$, for each $p \in V$, {\em $B_p$ is itself an integral manifold of~$D$}, by Lemma~\ref{lemma-jgd}. Whence $\bar \psi (U \times V)$ is an open neighborhood of $g_0$ contained in~$G^D$.

\subsection[Proof that $G^D \subset G$ is closed]{Proof that $\boldsymbol{G^D \subset G}$ is closed}
That $G^D \subset G$ is closed follows from a general property of involutive $n$-plane f\/ields: The torsion tensor~$\tau $, a~section of $\wedge^2 D^*\otimes (TM/D)$ def\/ined by $\tau(V_1,V_2)=[V_1,V_2] \modulo D$, must vanish on~$G^D $. Since $\tau $ is continuous, the complement of $G^D $ in~$G$ is therefore open.

Since $G^D$ is simultaneously open and closed in $G$, it is a union of connected components. Since the identity bisection integrates any Cartan connection $D$ on~$G$, we have $G^{0} \subset G$. This completes the proof of part~(2) of Theorem~\ref{theorem1.2}.

\appendix
\section{The canonical Cartan connection for Riemannian geometry}
Here we sketch the construction of the canonical Cartan connection associated with every Riemannian manifold $M$. Details will be given elsewhere. The detailed construction of the corresponding inf\/initesimal connection appears in~\cite{Blaom_12}. For simplicity, we shall take `isometry' to mean `orientation preserving isometry'. For then the transitive Lie groupoid $G \subset J^1 (M \times M) $ def\/ined in Section~\ref{reg} will have connected isotropy groups and connected source f\/ibres ($M$ is assumed to be connected).

Def\/ine the {\it prolongation} $\mathrm p G \subset J^1 G$ of $G$ by ${\mathrm p}G=J^2(M \times M)\cap J^1 G$. Here $J^2(M \times M)$ denotes the Lie groupoid of two-jets of local bisections of $M \times M$, regarded as a subgroupoid of $J^1 (J^1 (M \times M))$.

\begin{Lemma}\label{lemmaA.1}
 The natural projection $\mathrm p G \rightarrow G$ is an isomorphism.
\end{Lemma}

Supposing the lemma to be true, we let $S \colon G \rightarrow J^1 G$ be the unique right-inverse for the natural projection $J^1 G \rightarrow G$ that has ${\mathrm p} G$ as its image. Then $S$ is a~Cartan connection def\/ining an $n$-plane f\/ield~$D$ on~$G$.
\begin{Proposition}
 A local transformation $\phi $ of $M$ is an isometry if and only if
 \begin{gather}
 \phi (m)= \beta(b(m))\label{curio}
 \end{gather}
 for some local bisection $b$ of $G$ integrating $D$.
\end{Proposition}

Here $\beta \colon G \rightarrow M$ denotes the target projection.

\begin{proof}
The local transformations $\phi $ of $M$ that are isometries are in one-to-one correspondence with the holonomic bisections of $G$ (those bisections $b$ that are of the form $b=J^1 \phi $, for some bisection~$\phi $ of $M \times M$). In one direction the correspondence is given by f\/irst-order extension, \mbox{$\phi \mapsto J^1 \phi $}. In the other direction we pass from a bisection $b$ to a~local transformation~$\phi $ \mbox{using~\eqref{curio}}. But it is a fact that a~bisection $b$ of $G$ is holonomic if and only if {\em its} extension $J^1 b$ is a bisection of~${\mathrm p} G$, i.e., if and only if $J^1_m b \in S(b(m))$ at each $m$ in the domain of~$b$. Or, in other words, viewing~$b$ as a~submanifold of~$G$, if and only if $b$ is an integral manifold of~$D$.
\end{proof}

\begin{proof}[Proof of Lemma~\ref{lemmaA.1}] The corresponding inf\/initesimal statement is true. That is, the natural projection ${\mathrm p}{\mathfrak g} \rightarrow {\mathfrak g} $, where $\mathrm p {\mathfrak g}\subset J^1 {\mathfrak g} $ is def\/ined by $\mathrm p {\mathfrak g} :=J^2(TM) \cap J^1 {\mathfrak g}$, is an isomorphism. This follows from the exactness of a natural sequence
 \begin{gather*}
 0 \rightarrow \kernel \delta \rightarrow \mathrm p {\mathfrak g} \rightarrow {\mathfrak g} \rightarrow \cokernel \delta,
 \end{gather*}
where $\delta \colon T^*\!M \otimes {\mathfrak h} \rightarrow \wedge^2 T^*\!M \otimes TM$ is Spencer's coboundary operator and ${\mathfrak h} \subset {\mathfrak g} $ is the kernel of the anchor of~${\mathfrak g} $ (the Lie algebra bundle of skew-symmetric endomorphisms of tangent spaces in the present case). For details see~\cite{Blaom_12}.

To show that ${\mathrm p} G \rightarrow G$ is injective it suf\/f\/ices to show that the pre-image $\Sigma \subset {\mathrm p} G$ of the subgroup of identity elements is trivial, i.e., consists only of identity elements. The Lie algebroid of $\Sigma $ is trivial, because it is the kernel of ${\mathrm p} {\mathfrak g} \rightarrow {\mathfrak g} $. That the totally intransitive Lie groupoid $\Sigma $ itself is trivial follows from the fact that its f\/ibres are simply-connected Abelian Lie groups, as is not too hard to show.

Since ${\mathrm p} {\mathfrak g} \rightarrow {\mathfrak g} $ is surjective, the morphism ${\mathrm p} G \rightarrow G$ is a~submersion. Since its image $I$ is a wide subgroupoid of $G$, this image must be open and closed. Since $G$ has source-connected f\/ibres, $G$ is connected, implying $I=G$.
\end{proof}

\pdfbookmark[1]{References}{ref}
\LastPageEnding


\begin{thebibliography}{99}
\footnotesize\itemsep=0pt

\bibitem{Armstrong_07}
Armstrong S., Note on pre-{C}ourant algebroid structures for parabolic
 geometries, \href{http://arxiv.org/abs/0709.0919}{arXiv:0709.0919}.

\bibitem{Armstrong_Lu_12}
Armstrong S., Lu R., Courant algebroids in parabolic geometry,
 \href{http://arxiv.org/abs/1112.6425}{arXiv:1112.6425}.

\bibitem{Blaom_06}
Blaom A.D., Geometric structures as deformed inf\/initesimal symmetries,
 \href{http://dx.doi.org/10.1090/S0002-9947-06-04057-8}{\textit{Trans. Amer. Math. Soc.}} \textbf{358} (2006), 3651--3671,
 \href{http://arxiv.org/abs/math.DG/0404313}{math.DG/0404313}.

\bibitem{Blaom_12}
Blaom A.D., Lie algebroids and {C}artan's method of equivalence, \href{http://dx.doi.org/10.1090/S0002-9947-2012-05441-9}{\textit{Trans.
 Amer. Math. Soc.}} \textbf{364} (2012), 3071--3135, \href{http://arxiv.org/abs/math.DG/0509071}{math.DG/0509071}.

\bibitem{Blaom_13}
Blaom A.D., The inf\/initesimalization and reconstruction of locally homogeneous
 manifolds, \href{http://dx.doi.org/10.3842/SIGMA.2013.074}{\textit{SIGMA}} \textbf{9} (2013), 074, 19~pages,
 \href{http://arxiv.org/abs/1304.7838}{arXiv:1304.7838}.

\bibitem{Blaom_16}
Blaom A.D., Pseudogroups via pseudoactions: unifying local, global, and
 inf\/initesimal symmetry, \textit{J.~Lie Theory} \textbf{26} (2016), 535--565,
 \href{http://arxiv.org/abs/1410.6981}{arXiv:1410.6981}.

\bibitem{CannasdaSilva_Weinstein_99}
Cannas~da Silva A., Weinstein A., Geometric models for noncommutative algebras,
 \textit{Berkeley Mathematics Lecture Notes}, Vol.~10, Amer. Math. Soc.,
 Providence, RI, Berkeley Center for Pure and Applied Mathematics, Berkeley,
 CA, 1999.

\bibitem{Cap_Slovak_09}
{\v{C}}ap A., Slov{\'a}k J., Parabolic geometries.~{I}.~Background and general
 theory, \href{http://dx.doi.org/10.1090/surv/154}{\textit{Mathematical Surveys and Monographs}}, Vol.~154, Amer. Math.
 Soc., Providence, RI, 2009.

\bibitem{Cartan_10}
Cartan E., Les syst\`emes de {P}faf\/f, \`a cinq variables et les \'equations aux
 d\'eriv\'ees partielles du second ordre, \textit{Ann. Sci. \'Ecole Norm.
 Sup.~(3)} \textbf{27} (1910), 109--192.

\bibitem{Courant_90}
Courant T.J., Dirac manifolds, \href{http://dx.doi.org/10.2307/2001258}{\textit{Trans. Amer. Math. Soc.}} \textbf{319}
 (1990), 631--661.

\bibitem{Crainic_03}
Crainic M., Dif\/ferentiable and algebroid cohomology, van {E}st isomorphisms,
 and characteristic classes, \href{http://dx.doi.org/10.1007/s00014-001-0766-9}{\textit{Comment. Math. Helv.}} \textbf{78} (2003),
 681--721, \href{http://arxiv.org/abs/math.DG/0008064}{math.DG/0008064}.

\bibitem{Crainic_Fernandes_11}
Crainic M., Fernandes R.L., Lectures on integrability of {L}ie brackets, in
 Lectures on {P}oisson geometry, \textit{Geom. Topol. Monogr.}, Vol.~17, Geom.
 Topol. Publ., Coventry, 2011, 1--107, \href{http://arxiv.org/abs/math.DG/0611259}{math.DG/0611259}.

\bibitem{Crainic_Salazar_15}
Crainic M., Salazar M.A., Jacobi structures and {S}pencer operators,
 \href{http://dx.doi.org/10.1016/j.matpur.2014.04.012}{\textit{J.~Math. Pures Appl.}} \textbf{103} (2015), 504--521,
 \href{http://arxiv.org/abs/1309.6156}{arXiv:1309.6156}.

\bibitem{Crainic_etal_15}
Crainic M., Salazar M.A., Struchiner I., Multiplicative forms and {S}pencer
 operators, \href{http://dx.doi.org/10.1007/s00209-014-1398-z}{\textit{Math.~Z.}} \textbf{279} (2015), 939--979,
 \href{http://arxiv.org/abs/1210.2277}{arXiv:1210.2277}.

\bibitem{Crampin_Saunders_16}
Crampin M., Saunders D., Cartan geometries and their symmetries. A {L}ie
 algebroid approach, \href{http://dx.doi.org/10.2991/978-94-6239-192-5}{\textit{Atlantis Studies in Variational Geometry}},
 Vol.~4, Atlantis Press, Paris, 2016.

\bibitem{del_Hoyo_Fernandes_16}
del Hoyo M.L., Fernandes R.L., Riemannian metrics on dif\/ferentiable stacks, \href{http://arxiv.org/abs/1601.05616}{arXiv:1601.05616}.

\bibitem{Dufour_Nguyen_05}
Dufour J.P., Zung N.T., Poisson structures and their normal forms,
 \textit{Progress in Mathematics}, Vol.~242, Birkh\"auser Verlag, Basel, 2005.

\bibitem{Ehresmann_51}
Ehresmann C., Les connexions inf\/init\'esimales dans un espace f\/ibr\'e
 dif\/f\'erentiable, in Colloque de topologie (espaces f\/ibr\'es), {B}ruxelles,
 1950, Georges Thone, Li\`ege, Masson et Cie., Paris, 1951, 29--55.

\bibitem{Hitchin_03}
Hitchin N., Generalized {C}alabi--{Y}au manifolds, \href{http://dx.doi.org/10.1093/qjmath/54.3.281}{\textit{Q.~J.~Math.}}
 \textbf{54} (2003), 281--308, \href{http://arxiv.org/abs/math.DG/0209099}{math.DG/0209099}.

\bibitem{Kirillov_76a}
Kirillov A.A., Local {L}ie algebras, \href{http://dx.doi.org/10.1070/RM1976v031n04ABEH001556}{\textit{Russ. Math. Surv.}} \textbf{31}
 (1976), no.~4, 55--76.

\bibitem{Kirillov_77}
Kirillov A.A., Letter to the editors: {C}orrection to ``{L}ocal {L}ie
 algebras'' (\textit{Russ. Math. Surv.} {\bf 31} (1976), no.~4, 55--76),
 \textit{Russ. Math. Surv.} \textbf{32} (1977), no.~1, 268.

\bibitem{Lichnerowicz_78}
Lichnerowicz A., Les vari\'et\'es de {J}acobi et leurs alg\`ebres de {L}ie
 associ\'ees, \textit{J.~Math. Pures Appl.} \textbf{57} (1978), 453--488.

\bibitem{Mackenzie_05}
Mackenzie K.C.H., General theory of {L}ie groupoids and {L}ie algebroids,
 \href{http://dx.doi.org/10.1017/CBO9781107325883}{\textit{London Mathematical Society Lecture Note Series}}, Vol.~213, Cambridge
 University Press, Cambridge, 2005.

\bibitem{Morimoto_93}
Morimoto T., Geometric structures on f\/iltered manifolds, \href{http://dx.doi.org/10.14492/hokmj/1381413178}{\textit{Hokkaido
 Math.~J.}} \textbf{22} (1993), 263--347.

\bibitem{Salazar_13}
Salazar M.A., Pfaf\/f\/ian groupoids, Ph.D.~Thesis, University of Utrecht, The
 Netherlands, 2013, \href{http://arxiv.org/abs/1306.1164}{arXiv:1306.1164}.

\bibitem{Sharpe_97}
Sharpe R.W., Dif\/ferential geometry: {C}artan's generalization of {K}lein's
 {E}rlangen program, \textit{Graduate Texts in Mathematics}, Vol.~166,
 Springer-Verlag, New York, 1997.

\bibitem{XuXiaomeng_13}
Xu X., Twisted {C}ourant algebroids and coisotropic {C}artan geometries,
 \href{http://dx.doi.org/10.1016/j.geomphys.2014.03.002}{\textit{J.~Geom. Phys.}} \textbf{82} (2014), 124--131, \href{http://arxiv.org/abs/1206.2282}{arXiv:1206.2282}.

\end{thebibliography}
\end{document}